\documentclass[a4paper,11pt]{amsart}

\usepackage{amssymb,amscd,amsthm,mathrsfs}
\usepackage[ansinew]{inputenc}
\usepackage[centertags]{amsmath}
\usepackage{graphicx,psfrag}

\newcommand{\la}{\lambda}
\newcommand{\ep}{\epsilon}
\newcommand{\N}{\mbox{$\mathbb{N}$}}
\newcommand{\Z}{\mbox{$\mathbb{Z}$}}
\newcommand{\C}{\mbox{$\mathcal{C}$}}
\newcommand{\R}{\mbox{$\mathbb{R}$}}

\newcommand{\T}{\mbox{$\mathbb{T}$}}

\newcommand{\F}{\mathcal{F}}
\newcommand{\h}{\tilde{h}}

\newtheorem*{teo*}{Theorem}
\newtheorem*{mainteo}{Main Theorem}

\newtheorem{teo}{Theorem}[subsection]

\newtheorem{cor}{Corollary}[subsection]
\newtheorem{lema}{Lemma}[subsection]
\newtheorem{prop}{Proposition}[subsection]

\newcommand{\bi}{\begin{itemize}}
\newcommand{\ei}{\end{itemize}}

\theoremstyle{definition}
\newtheorem{defi}{Definition}[section]
\theoremstyle{remark}
\newtheorem{obs}[]{Remark}[section]

\author{Alejandro Passeggi, Mart\'{\i}n Sambarino}
\address{CMAT, Facultad de Ciencias, Universidad de la Rep\'ublica, Uruguay}
\email{alepasseggi@cmat.edu.uy} \email{samba@cmat.edu.uy}

\title[Examples of minimal diffeomorphisms on $\T^{2}$]{Examples of minimal diffeomorphisms on $\T^{2}$ semiconjugated to an ergodic translation}

\begin{document}

\begin{abstract}

We prove that for every $\epsilon>0$ there exists a minimal
diffeomorphism $f:\T^{2}\rightarrow\T^{2}$ of class $C^{3-\epsilon}$
and semiconjugate to an ergodic traslation, with the following
properties: zero entropy, sensitivity with respect to initial
conditions, and Li-Yorke chaos. These examples are obtained through
the holonomy of the unstable foliation of Ma\~{n}\'{e}'s example of derived
from Anosov diffeomorphism on $\T^3.$






\end{abstract}
\maketitle

\section{Introduction.}

The classical result of Denjoy (\cite{D}) can be stated as follows:
if $f:S^1\to S^1$ is a $C^2$ diffeomorphism semiconjugated to an
ergodic rotation then it is indeed conjugated to it. One may ask to
what extent Denjoy´s theory on $S^1$ can be extended to higher
dimensional tori. In particular one may ask: Does there exist $r$ so
that if $f:\T^2\to \T^2$ is a $C^r$ diffeomorphism semiconjugated to
an ergodic translation then $f$ is conjugated to it? This seems to
be a very difficult question. Nevertheless, KAM theory provides a
particular result when $f$ is close to an ergodic translation of
diophantine type. There are also some indications that if the above
question has a positive answer, then $r=3$ (see for instance
\cite{McS} for $C^{3-\ep}$ examples with wandering domains and
\cite{NS}).

One may ask if there are extra restrictions on the differentiability
class of a Denjoy type map $f:\T^2\to\T^2$ (i.e. semiconjugataed
-but not conjugated- to an ergodic traslation) if we also assume
that $f$ is minimal (i.e. every orbit is dense). In this paper we
give examples of this type of class $C^{3-\ep}$ for any $\ep>0.$ If
we denote by $h$ the semiconjugacy then, if $f$ is minimal, the
fibers $h^{-1}(x)$ have empty interior and we may ask how they
look-like. We show that the fibers are points or arcs. The first
(topological) example of this kind was given by M. Rees in \cite{R1}
in order to construct a non-distal (but point distal) homeomorphism
of $\T^2.$ Recall that $f$ is \textit{non-distal} if there exist
$x\neq y$ such that $\inf_{n\in\Z}\{dist(f^n(x),f^n(y))\}=0$ and $f$
is \textit{point distal} if there exists $x$ such that for any
$y\neq x, \inf_{n\in\Z}\{dist(f^n(x),f^n(y))\}>0.$

The dynamics on  the nontrivial fibers by the action of the map $f$
has a chaotic flavour: it compress them to an arc of lenght
arbitrarily small and then stretches to an arc of fixed length and
then compresses them and so on. Morevoer, there are some nontrivial
fibers that the the latter holds for any subarc. These imply
interesting properties: \textit{sensitivity with respect to initial
conditions} (that is, there exists some $\ep>0$ so that for any
$x\in \T^2$ and any nieghborhood $U(x)$ there exists $y\in U$ and
$n>0$ such that $dist(f^n(x),f^n(y))>\ep)$ and \textit{Li-Yorke
chaos} (i.e, a noncountable \textit{scrambled} set where any pair of
points $x\neq y$ in this set verify $\liminf_n
dist(f^n(x),f^n(y))=0$ and $\limsup_n dist(f^n(x),f^n(y))>0$). One may ask whether 
these examples also have distributional chaos and whether deserve to be called chaotic (see \cite{O} for a discussion  on the subject).

It is also interesting to ask about the ergodic properties of these
Denjoy type maps.  Unfortunately, our examples are simple from this
point of view: they have just one invariant measure, i.e., are
\textit{uniquely ergodic}. Neverhteless we post the question: does
there exist minimal diffeomorphism semiconjugated to an ergodic
translation not uniquely ergodic? For homeomorphisms the answer is
positive (see \cite{R2}).

Our result is the following:

\begin{mainteo}\label{mainteo}
For all $r\in [1,3)$ there exists a diffeomorphism $f:\T^2\to\T^2$
of class $C^r$ such that:
\begin{itemize}
\item $f$ is minimal.
\item $f$ is isotopic and semiconjugated (but not conjugated) to an ergodic traslation. If we denote by $h$ the semiconjugacy, then $h^{-1}(x)$ is either a point or an arc. Moreover, there are uncountable points $x$ such that $h^{-1}(x)$ is a nontrivial arc.

\item $f$ preserves a minimal and  invariant foliation with one dimensional $C^1$  leaves. The fibers $h^{-1}(x)$ are contained in the leaves of this foliation.
    \item $f$ has zero entropy.
\item $f$ has sensitivity with respect to initial conditions
\item exhibits Li-York chaos.
\item $f$ is point-distal and non-distal.
\item $f$ is uniquely ergodic.
\end{itemize}
\end{mainteo}

The proof of our theorem is inspired by \cite{McS}. There, the
examples are constructed through the holonmy map from a cross
section to itself of the unstable foliation of a derived from Anosov
diffeomorphism obtained through a Hopf´s bifurcation. In this paper
we use instead Ma\~{n}\'{e}'s example of derived from Anosov diffeormphism
(\cite{M1}) where we prove that the unstable foliation is minimal.
However, there is a main difference with \cite{McS}: while  there
the starting linear Anosov map is fixed and for any $\ep$ a
modification is taken so that the unstable foliation is $C^{3-\ep}$
we have to do it the other way around, that is, given $\ep>0$ we
have to find the  linear Anosov map to begin with so that a
modification can be done such that the unstable foliation of the
resulting diffeomorphism is $C^{3-\ep}.$ This modification also
includes the existence of periodic points of different unstable
indices and the existence of transversal homoclinic points
associated to them.

The paper is organized as follows: in Section \ref{secmane} we give
our construction of Ma\~{n}\'{e}'s derived from Anosov diffeomorphism and we
prove the minimality of the unstable foliation (see Section
\ref{secmini}) and the minimality of the central foliation through
the semiconjugacy with the linear Anosov map (see Section
\ref{secsemiup}); in Section \ref{secholonomy} we give the
topological version of our main result and in Section \ref{secdif}
we prove the differentiability of the unstable foliation through the
$C^r$ Section Theorem (\cite{HPS}).

\textbf{Acknowledgements:} We would like thanks Alejandro Kocsard,
Andres Koropecki and specially to Tobias J\"{a}ger for useful
conversations and comments.

\section{On Ma\~{n}e's Derived from Anosov diffeomorphism}\label{secmane}

In \cite{M1} R.Ma\~{n}\'{e} construct an example on $\T^{3}$ which is
robustly-transitive but not Anosov. This is known as Ma\~{n}e's Derive
from Anosov diffeomorphisms due to the construction: it begins with
an Anosov linear map on $\T^3$ with partially hyperbolic structure
$E^{s}\oplus E^c\oplus E^u$ and modifies it in a neighborhood of the
fixed point in order to change the unstable index of it (and
preserving the partially hyperbolic structure). See Figure
\ref{indices}.

\begin{figure}[ht]\begin{center}

\psfrag{ss}{$E_{\lambda_s}$}\psfrag{uu}{$E_{\lambda_u}$}\psfrag{c}{$E_{\lambda_c}$}
\psfrag{0}{$p$}\psfrag{1}{$q_1$}\psfrag{2}{$q_2$}
\psfrag{ssm}{$E^{ss}$}\psfrag{cm}{$E^{c}$}\psfrag{uum}{$E^{uu}$}\psfrag{I}{$I_p$}

\includegraphics[height=4cm]{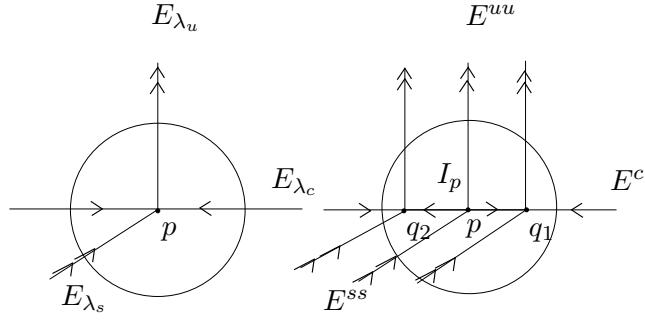}
\caption{Modfication.}\label{indices}
\end{center}\end{figure}

Let us be more precise. Let $\T^{3}=\R^3/_{\Z^3}$ be the three
dimensional torus and  denote by $\pi:\R^{3}\rightarrow\T^{3}$ the
canonical projection, and set  $p=\pi(0)$.

Consider $B\in SL(3,\Z)$ with eigenvalues
$0<\lambda_s<\lambda_c<1<\lambda_u$ and denote also by  $B$ the
induced Linear Anosov system on $\T^{3}$ with hyperbolic structure
$T\T^3=E^s\oplus E^c\oplus E^u$ (corresponding to the eigenspaces
associated to $\lambda_s,\lambda_c$ and $\lambda_u).$ For the sake
of simplicity to do our calculations we will define an Euclidean
metric on $\R^3$ so that $E^s_B, E^c_B$ and $E^u_B$ are mutally
orthogonal.

Let $\rho$ be small and consider $B(p,\rho)$ the ball centered at
$p.$ Let $Z:\R\rightarrow\R$ be a $C^{\infty}$ bump function such
that $Z(0)=1$, $sop[Z]\subset (-\frac{\rho}{2},\frac{\rho}{2})$ and
$|Z'(z)|<\frac{4}{\rho}$. (See Figure \ref{Z})

\begin{figure}[ht]\begin{center}

\psfrag{Z}{$Z$} \psfrag{1}{$1$} \psfrag{ro}{$\frac{\rho}{2}$}
\psfrag{-ro}{$-\frac{\rho}{2}$}

\includegraphics[height=4cm]{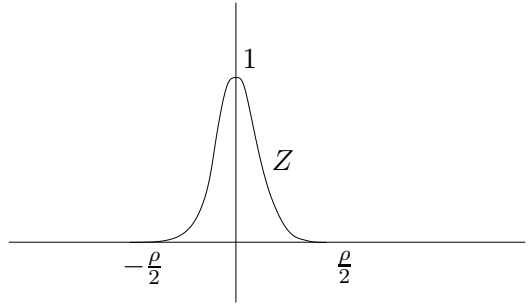}
\caption{The bump function $Z$.}\label{Z}
\end{center}\end{figure}

For our construction of the Ma\~{n}e's Derived from Anosov \footnote{Our
contruction is slighty different because we need to keep control of
the relation between $E^s$ and $E^c$ to obtain higher
differentiability of the unstable foliation. In particular the
central foliation is not kept unchanged.} we need an auxiliary
function as in the next lemma.

\begin{lema}\label{mod}

For all $k>0$ arbitrarily small there exist a function
$\beta_k:\R^{+}\cup \{0\}\rightarrow\R$ such that:

\begin{enumerate}
\item $\beta_k$ is $C^{\infty}$, decreasing and such that $-k
\leq \beta_k'(t)t \leq 0$.

\item $\beta_k$ is suuported in $[0.k]$, i.e. $supp[\beta_k]\subset [0,k].$
\item $\lambda_s + \beta_k(0)< 1 < \lambda_c + \beta_k(0)<1+k$.

\end{enumerate}

\end{lema}

\begin{figure}[ht]\begin{center}

\psfrag{kt}{$y=\frac{k}{t}$} \psfrag{b}{$b$} \psfrag{r0}{$r_0$}
\psfrag{psi}{$\psi$}

\includegraphics[height=5cm]{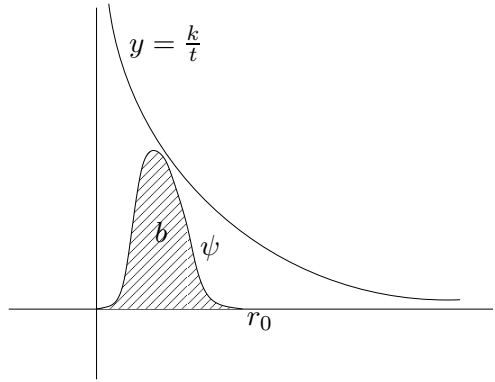}
\caption{The function $\psi$.}\label{psi}
\end{center}\end{figure}

\begin{proof}

We may assume that $0<k<\lambda_c-\lambda_s$ and take $b$ such that
$1-\lambda_c<b<1-\lambda_c+k.$ Let $r_0<k$. Since
$\int_0^{r_0}\frac{k}{t}dt$ is divergent we may find a $C^\infty$
non negative function $\psi$ with support in $[0,r_0]$ such that
$\int_0^{r_0}\psi(t)dt=b$ and $\psi(t)\le \frac{k}{t}$ (in other
words the graph of $\psi$ is below the graph of $h(t)=\frac{k}{t}.$

Define
$$\beta_k(t)=b-\int_0^t\psi(s)ds.$$

This function satisfies the lemma.
\end{proof}

Finally, define $g_{B,k}: \T^3\to\T^3$ defined by:

\begin{equation}\label{id}
g_{B,k}(\xi)=B(\xi) \,\,\,\,\,\mbox{ for }\,\,\,\,\, \xi\notin
B(p,\rho)
\end{equation}
and for $\xi\in B(p,\rho)$ in local coordinates with respect to
$E^s_B\oplus E^c_B\oplus E^u_B, \xi=(x,y,z)$
\begin{equation}\label{g}
g_{B,k}(\xi)=(\lambda_s x,\lambda_c y,\lambda_u
z)+Z(z)\beta_{k}(r)(x,y,0)
\end{equation}

where $r=x^{2}+y^{2}$.

\begin{prop}\label{DA}
If $k$ is sufficiently small, then $g_{B,k}:\T^3\to\T^3$ defined
above is a diffeomorphism with partially hyperbolic structure
\footnote{We remark that it is not absolutely partially hyperbolic.}
$T\T^3=E^s_{g_{B,k}}\oplus E^c_{g_{B,k}}\oplus E^u_{g_{B,k}}$ where
$E^s_{g_{B,k}}$ is uniformly contracting and $E^u_{g_{B,k}}$ is
uniformly expanding. Moreover, given cones $C^s, C^c$ and $C^u$
around $E^s_B, E^c_B$ and $E^u_B$ respectively we have that
$E^s_{g_{B,k}}\in C^s, E^c_{g_{B,k}}\in C^c$ and $E^u_{g_{B,k}}\in
C^u.$ Furthermore, the same is true for any $g$ in any sufficiently
small $C^1$ neighborhood $\mathcal{U}$ of $g_{B,k}$.
\end{prop}

\begin{proof}
First of all, the $C^0$ distance between $g_{B,k}$ and $B$ is
smaller than $\sqrt{k}$ and hence (assuming $k$ small) we conclude
that $g_{B,k}$ is a differentiable homeomorphism. To avoid notation,
set $g=g_{B,k}$ for the time being.

For $\xi\notin B(p,\rho)$ we have $dg_\xi=B.$ For $\xi\in B(p,\rho)$
we have (with respect to the decomposition $E^s\oplus E^c\oplus
E^u$)
\begin{equation}
dg_{\xi}= \left( \begin{array}{ccc}
\lambda_s + Z(z)(\beta(r)+\beta'(r)2x^2) & Z(z)\beta'(r)2xy & Z'(z)\beta(r)x \\
Z(z)\beta'(r)2xy& \lambda_c + Z(z)(\beta(r)+\beta'(r)2y^2) & Z'(z)\beta(r)y \\
0 & 0 & \lambda_u \end{array} \right)
\end{equation}

We may write $dg_\xi=A_\xi+M_\xi$ (and agreeing that $Z$ and $\beta$
are identically zero outside $B(p,\rho)$) where
\begin{equation}
A_{\xi}= \left( \begin{array}{ccc}
\lambda_s + Z(z)\beta(r) & 0 & 0\\
0 & \lambda_c + Z(z)\beta(r) & 0 \\
0 & 0 & \lambda_u \end{array} \right)
\end{equation}
and
\begin{equation}
M_\xi=\left( \begin{array}{ccc}
 Z(z)\beta'(r)2x^2 & Z(z)\beta'(r)2xy & Z'(z)\beta(r)x \\
Z(z)\beta'(r)2xy&  Z(z)\beta'(r)2y^2 & Z'(z)\beta(r)y \\
0 & 0 & 0 \end{array} \right)
\end{equation}
Since $|\beta'(r)r|\le k$ it is straighborward to check that
$\|M_\xi\|\leq \max\{2k, 8\beta(0)\sqrt{k}/\rho\}.$ Therefore,
choosing $k$ arbitrarily small we get that $\|M_\xi\|$ is also
arbitrarily small. Since the co-norm ($=\|A_\xi^{-1}\|^{-1}$)  of
$A_\xi$ is bounded away from zero we have that $dg_\xi$ is an
isomorphism and hence $g$ a diffeomorphism. On the other hand,
$A_\xi(E^j_B)=E^j_B,\;j=s,c,u$ and
\begin{itemize}
\item $\lambda_s\le \|A_{\xi/E^s_B}\|\le \lambda_s+\beta_k(0)<1$
\item $\lambda_c\le \|A_{\xi/E^c_B}\|\le \lambda_c+\beta_k(0)<1+k$
\item $\frac{\|A_{\xi/E^s_B}\|}{\|A_{\xi/E^c_B}\|}\le \frac{\lambda_s}{\lambda_c}<1.$
\item $\|A^-1_{\xi/E^u_B}\|\le \lambda_u^{-1}.$
\end{itemize}
 From this it is easy to conclude the proof of the proposition,
 taking $k$ sufficiently small (and so $\|M_\xi\|$ is sufficiently
 small) and taking $\mathcal U$ sufficiently small.

\end{proof}

For $g_{B,k}:\T^{3}\rightarrow\T^{3}$ with $k$ small and
$g\in\mathcal U(g_{B,k})$ so that the above proposition applies we
set:

\begin{itemize}

\item $\lambda_{s}(g)(\xi)=\|
dg_{\xi/E_g^{s}}\|$, and
$\lambda_{s}(g)=max_{\xi\in\T^{3}}\{\lambda_{s}(g)(\xi)\}$.

\item $\lambda_c(g)(\xi)=\|
dg_{\xi/E_g^{c}}\|$, and
$\lambda_c(g)=max_{\xi\in\T^{3}}\{\lambda_c(g)(\xi)\}$.

\item $\lambda_{u}(g)(\xi)=\|
dg_{\xi/E_g^{u}}\|$, and
$\lambda_{u}(g)=min_{\xi\in\T^{3}}\{\lambda_{u}(g)(\xi)\}$.

\end{itemize}

\begin{obs}\label{epsilon}
 Notice that, given $\ep>0$ small,  the following conditions hold
for $g\in\mathcal U(g_{B,k})$ with $k$ and $\mathcal {U}$
sufficiently small:

\begin{enumerate}

\item
$0<\lambda_{s}(g)(\xi)<\lambda_c(g)(\xi)<\lambda_{u}(g)(\xi)$ for
all $\xi\in\T^{3}$.

\item $\lambda_c-\ep<\lambda_c(g)(\xi)$ for all $\xi\in\T^3.$

\item $\lambda_{s}(g)<\lambda_s+\beta(0)+\ep<1$.

\item $\lambda_s(g)(\xi)<\lambda_s+\ep$ for $\xi \in\T^3-
B(p,\rho).$
\item $\lambda_{c}(g)<\lambda_c+\beta(0)+\ep$.

\item $\lambda_{u}(g)>\la_u-\ep>1$ and
$\lambda_{u}(g)>\lambda_c+\beta(0)+\ep.$

\end{enumerate}
\end{obs}

Once we know that $g\in\mathcal U(g_{b,k})$ is partially hyperbolic,
by well known results (see \cite{HPS}) we get that the bundles
$E^s_g$ and $E^u_g$ uniquely integrate to foliations $\F^s_g$ and
$\F^u_g$ called the (strong) stable and unstable foliations
respectively.

Moreover, since $E^s_g$ and $E^u_g$ are contained in tiny cones
around $E^s_B$ and $E^u_B$ we conclude that $\F^s_g$ and $\F^u_g$
are quasi isometric. Nevertheless, since $g$ is not absolutely
partially hyperbolic the result in  \cite{B} (see also \cite{BBI})
does not apply to prove that $E^c_g$ is uniquely integrable.
Recently, R. Potrie \cite{Po} has extended the results in \cite{BBI}
to the non absolutely partially hyperbolic setting and we conclude
that $E^c_g$ is uniquely integrable. However, for our particular
case we can give a direct proof of the unique integration of $E^c_g$
in the spirit of \cite{B} (see Section \ref{integrability}). We denote by $\F^c_g$ this central
foliation, and therefore the bundles $E^s_g\oplus E^c_g$ and
$E^c_g\oplus E^u_g$ are uniquely integrable and leads to the central
stable and central unstable foliations. We also remark that in the
particular case $g=g_{B,k}$ it holds that $E^s_g\oplus
E^c_g=E^s_B\oplus E^c_B$ and so the central stable foliation of
$g_{B,k}$ coincides with the two dimensional stable foliation of
$B.$

Also, in the following subsections we are going to study properties of the
invariant foliations and also consequences of the semiconjugacy with
the linear Anosov map. These results are fundamental for our
purposes.

\begin{teo}\label{folmin}
For all $k$ sufficiently small and $\mathcal U(g_{B,k})$
sufficiently small as well, the central bundle $E^c_g$ uniquely
integrates to an invariant foliation $\F^c_g.$ Furthermore, the
central and unstable foliations $ \F^c_g, \F^u_g$ of
$g\in\mathcal U(g_{B,k})$ are minimal, i.e., all leaves are dense.
\end{teo}

The minimality of $\F^u_g$ can be obtained from \cite{PS} and the
minimality of $\F^c_g$ will follows from the semiconjugacy with the
linear Anosov map.
We are going to give a complete proof of the theorem in Sections \ref{integrability},
\ref{secmini} (see Theorem \ref{uminimal}) and \ref{secsemiup} (see
Corollary \ref{cmini}).

Since every $g\in \mathcal U(g_{B,k})$ is transitive (this follows from the minimality of $\F^u_g$) and have periodic  points of different indices then it follows that the set of diffeomorphisms having a nonhyperbolic periodic point  is dense in $\mathcal U(g_{B,k})$ (see \cite{M2}, \cite{A} and \cite{H2}). We have the following

\begin{cor}\label{homo}
Let $k$ and $\mathcal U(g_{B,k})$ be as in the above theorem. Then,
there exists $g\in\mathcal U(g_{B,k})$ of class $C^\infty$ such that
\begin{enumerate}
\item \label{transverse} $g$ has a
transversal homoclinic point associated to a periodic point of unstable index $2.$
\item There exists a non trivial arc $J$ such that, for some $m>0\, g^m_{/J}=id_{/J}$, that is, $J$ consists of periodic points of $g$ of the same period $m.$
\end{enumerate}

\end{cor}

\begin{proof}
Notice that for $g_{B,k}$ the fixed point $p=\pi(0)$ has unstable
index equals to $2$ since $dg_{B,k/E^c}=\la_c+\beta(0)>1.$  On the
other hand, since $\F^u_{g_{B,k}}(p)$ is dense (and hence
accumulates on $\F^s_{g_{B,k}}(p)$) by the Hayashi's connecting
lemma (see \cite{H1})) we can perturb $g_{B,k}$ (with support
disjoint from a ball at $p$) and find $g_1$ satisfying condition (\ref{transverse}). Any diffeomorphism close to $g_1$ will also satisfy (\ref{transverse}). Now, we may find a diffeomorphism arbitrarily close to $g_1$ having a non hyperbolic periodic point $q$ and so, by another arbitrarily small perturbation we can transform this nonhyperbolic periodic point $q$ into an arc $J$ of periodic points and find $g$ as in the statement.

\end{proof}

\subsection{Unique integrability of the bundle
$E^c_g.$}\label{integrability}

We first recall that a foliation $\F$ in $\R^3$ is quasi-isometric
if there exist $C,D$ positive numbers such that if $x,y$ belongs to
same leaf of the foliation, i.e. $y\in\F(x)=\F(y)$ then
$$d(x,y)\ge Cd_{\F}(x,y)-D$$
where $d_\F$ means the distance along the leaf of the foliation.

Denote by $\tilde{\F}^j_G, j=s,u$ the lifts to the universal cover
$\R^3$ of the stable and unstable foliations $\F^j_g, j=s,u$ for $g\in\mathcal U(g_{B,k}).$ These
foliation are quasi-isometric as we remarked before. In particular, this means that if we have two points $x,y$ in the same unstable leaf, then by future iteration, the rate of growth of $d(G^n(x),G^n(y))$ is the same as $d_{\F^u_G}(G^n(x),G^n(y)).$ And similarly in the past for points in the a stable leaf.

Now, assume by contradiction that the central bundle $E^c_g$ is not (locally) uniquely integrable (at some point, say $x$). This implies (see \cite{B}) that there exist two points $z,w$ such that (see Figure \ref{nonint})
\begin{itemize}
\item $z,w$ can be joined by a curve $J^c$ always tangent to $E^c_g.$
\item $z,w$ can be joined by union of two curves $J^s, J^u$ always tangent to $E^s_g$ and $E^u_g$ respectively (of course, one of them could be trivial).
\end{itemize}

\begin{figure}[ht]\begin{center}
\psfrag{w}{$w$}
\psfrag{x}{$x$}
\psfrag{z}{$z$}
\psfrag{Jc}{$J^c$}
\psfrag{Js}{$J^s$} \psfrag{Ju}{$J^u$}

\includegraphics[height=4cm]{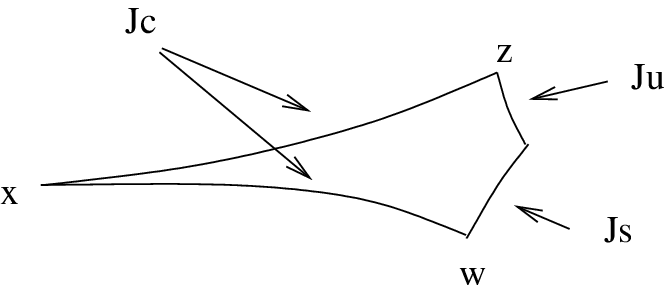}
\caption{}\label{nonint}
\end{center}\end{figure}

The same holds in the universal cover and we will argue there. If the curve $J^u$ is not trivial, then by future iteration we get that $d(G^n(z),G^n(w))$ grows at most with rate $\la_c(g)$ and on the other hand, by the quasi isometry of the unstable foliation we have that the rate of growth of $d(G^n(z),G^n(w))$ is the rate of growth of $G^n(J^u)$ (since the length of $G^n(J^s)$ decreases exponentially) which is at least $\la_u(g)$ which is bigger than $\la_c(g)$ (see remark \ref{epsilon}), a contradiction.

If $J^u$ is trivial, the argument is the same but more subtle and we have to do better estimations. Let $\epsilon$ and $\delta$ be small enough such that
$$\sigma:=(\la_s+\epsilon)^{-1}(1-\delta)>(\la_c-\epsilon)^{-1}$$
(recall that $\la_s,\la_c$ are eigenvalues of $B$). Now, choose $k,\rho$ and $\mathcal{U}(g_{B,k})$ small so that remark \ref{epsilon} applies and such that any curve tangent to $E^s_g$ of length at least one then the portion of it outside $B(p,\rho)$ is larger than $1-\delta.$

Now we are ready to go back to our situation of the points $z,w.$ Since $z,w$ are joined by a curve tangent to $E^c_g$ then,
$$d(G^{-n}(z),G^{-n}(w))\le (\la_c-\epsilon)^n\ell(J^c).$$
On the other hand, let $n_0$ be such that $G^{-n_0}(J^s)$ has length greater than one. Then
$$\ell(G^{-n}(G^{-n_0}(J^s))\ge (\la_s+\epsilon)^{-n}(1-\delta)^n=\sigma^n.$$
For $C$ and $D$ the constant of the quasi isometry of the stable foliation, for $n$ large enough we have that
$$C\sigma^n-D>(\la_c-\epsilon)^{n+n_0}\ell(J^c)$$
and so we get a contradiction:
\begin{eqnarray*}
d(G^{-n-n_0}(z),G^{-n-n_0}(w))&\le& (\la_c-\epsilon)^{n+n_0}\ell(J^c)\\&<& C\sigma^n-D\le C\ell(G^{-n}(G^{-n_0}(J^s))-D\\&\le& d(G^{-n-n_0}(z),G^{-n-n_0}(w))
\end{eqnarray*}

Thus, we have finished the proof of the unique integrability of $E^c$, i.e., the first part of Theorem \ref{folmin}.

\subsection{Minimality of the unstable foliation.}\label{secmini}

In this subsection we will prove that $\F^u_g$ is minimal for
$g\in\mathcal U(g_{B,k})$ for $k$ and $\mathcal U$ small enough. The
proof is based on the ideas and methods in \cite{PS}:

\begin{teo}\label{uminimal}
For all $k$ sufficiently small and $\mathcal U(g_{B,k})$
sufficiently small as well, the  unstable foliation $\F^u_g$ of
$g\in\mathcal U(g_{B,k})$ is minimal, i.e., all leaves are dense.
\end{teo}

\begin{proof}

 Recall that $0<\la_s<\la_c<1<\la_u$ are the eigenvalues of $B.$ Choose $\sigma, 1-(\la_c-\la_s)<\sigma<1.$ We may assume that $\rho$ (the radius of the ball centered at $p$ where the modification of $B$ is performed) is small so that any arc $I^s$ in $\F^s_B$ of length one has a subarc $I^s_1$ of length at least $1/3$ with empty intersection with $B(p,2\rho).$

Let $n_0$ so that
\begin{equation}\label{sigma}
\sigma^{-n_0}>3.
\end{equation}
Let $\epsilon, 0<\ep<\rho$ be such that $1-(\la_c-\la_s)+\ep<\sigma$
and
\begin{equation}\label{lambda}
\la:=\la_c(1+\ep)^{n_0}<1.
\end{equation}

Let us denote by $D^{cs}_g(x,\ep)$ a disc centered at $x$ and radius
$\ep$ in the central stable leaf through $x,$ $\F_g^{cs}(x).$

Now, we may assume that $k$ and $\mathcal U$ are so small so that
the following holds for $g\in\mathcal U(g_{B,k}):$

\begin{enumerate}\renewcommand{\theenumi}{\roman{enumi}}
\item $\la_s(g)<\sigma.$
\item \label{cs1}$\la_c(g)<1+\ep.$
\item \label{cs2} $\|dg_{/E^{cs}_g(\xi)}\|\le \la_c(1+\ep)$ if $\xi\notin B(p,\rho).$
\item Any arc $I^s$ of $\F^s_g$ of length at least one has a subarc $I^s_1$  of length at least $1/3$ with empty intersection with $B(p,2\rho)$.
\item \label{edense}Any leaf of $\F^u_g$ has nonempty intersection with $D^{cs}_g(x,\ep)$ for any $x$ (since $\F^u_B$ is minimal and for $k$ and $\mathcal U$ small the bundles $E^u_B$ and $E^u_g$ are close).
\end{enumerate}

Given $x\in\T^{3}$ let $I^s(x)$ an arc of length one such that $x\in
I^s(x)\subset \F^s_g(x).$ We know that there exists a subarc $I^s_1$
of length at least $1/3$ such that $I_1^s\cap B(p,2\rho)=\emptyset.$
Now, by \eqref{sigma}, we conclude that $g^{-n_o}(I^s_1)\subset
\F^s_g(g^{-n_0}(x)$ is an arc of length at least one. Therefore,
there exists a subarc $I^s_2\subset g^{-n_o}(I^s_1)$ of length at
least $1/3$ such that $I^s_2\cap B(p,2\rho)=\emptyset.$ Arguing by
induction, we conclude that for each $j\ge 1$ there exists
$I_{j+1}^s\subset g^{-n_0}(I^s_j)$ such that $I_{j+1}^s\cap
B(p,2\rho)=\emptyset.$

Define
$$z_x=\bigcap_{j\ge 1}g^{jn_0}(I_{j+1}).$$
Notice that
\begin{equation}\label{puntofuera}
z_x\in I^s(x)\,\,\,\mbox{ and }\,\,g^{-jn_0}(z_x)\notin
B(p,2\rho)\,\,\forall \,\,\, j\ge 0.
\end{equation}

In other words, in any arc of length one on any leaf of $\F^s_g$
there exists a point whose $g^{n_0}$-backward orbit never meets
$B(p,2\rho).$ Let $z=z_x$ be such a point and let $j\ge 1.$ Then we
have that $$D^{cs}_g(g^{-jn_0}(z),\ep)\cap B(p,\rho)=\emptyset$$ and
so, for any $y\in D^{cs}_g(g^{-jn_0}(z),\ep)$ we have that, by
\eqref{lambda}, \eqref{cs1} and \eqref{cs2}
$$\|dg^{n_0}_y\|\le \la_c(1+\ep)^{n_0}=\la<1$$
and therefore
\begin{equation}\label{csiterado}
g^{n_0}(D^{cs}_g(g^{-jn_0}(z),\ep))\subset
D^{cs}_g(g^{-(j-1)n_0}(z),\la \ep)
\end{equation}
and so, for any $1\le m\le j$ we have
$$g^{mn_0}(D^{cs}_g(g^{-jn_0}(z),\ep))\subset D^{cs}_g(g^{-(j-m)n_0}(z),\la^m \ep).$$

Now, we are ready to conclude the proof of the minimality of
$\F^u_g$ (for the argument see Figure \ref{arg}). Let $\xi\in\T^{3}$
and $U$ some open set in $\T^{3}$. We want to prove that
$$\F^u_g(\xi)\cap U\neq\emptyset.$$  Let $y\in U$, and  consider an
arc $J_y\subset F^s_g(y), J_y\subset U.$  There exists $m_0$
 so that $g^{-m_0}(J_y)$  has length grater than one. Let $z\in g^{-m_0}(J_y)$ the point constructed above,  and let $\mu$ be such that
\begin{equation}\label{mu}
g^{m_0}(D^{cs}_g(z,\mu))\subset U.
\end{equation}
Let $m_1$ be such that $\la^{m_1}\ep<\mu.$ From \eqref{csiterado} we
conclude  that
$$g^{m_1n_0}(D^{cs}_g(g^{-m_1n_0}(z),\ep)\subset D^{cs}_g(z,\mu).$$

\begin{figure}[ht]\begin{center}

\psfrag{U}{$U$} \psfrag{z}{$z$} \psfrag{Dz}{$D^{cs}(z,\mu)$}
\psfrag{gm0}{$g_{m_0}$} \psfrag{gm1n0}{$g^{m_1n_0}$}
\psfrag{Dcm1n0}{$D^{cs}(g^{-m_1n_0}(z),\ep)$}
\psfrag{Fs}{$\F^u_g(g^{-m_1n_0-m_0}(\xi))$}

\includegraphics[height=8cm]{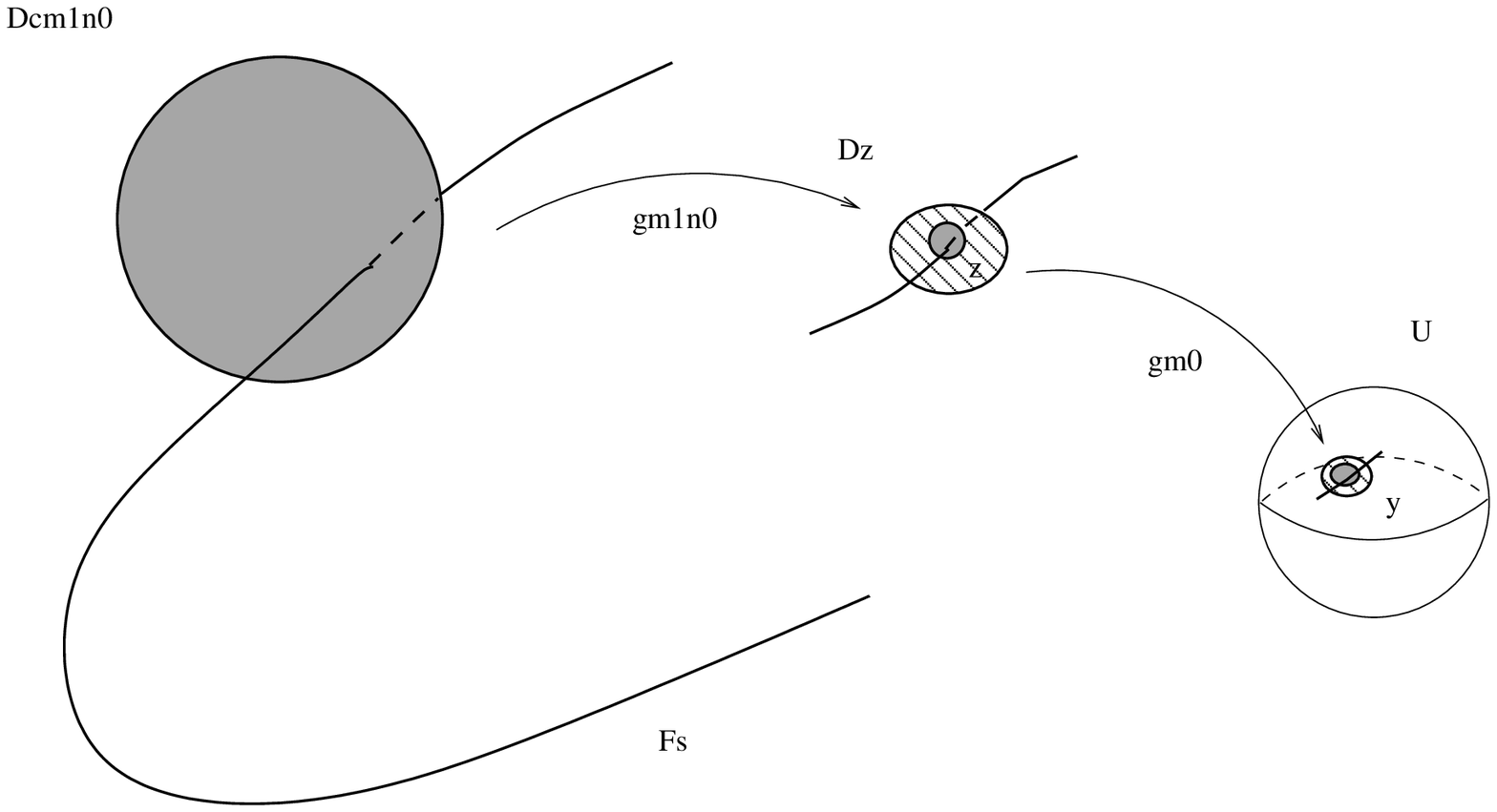}
\caption{}\label{arg}
\end{center}\end{figure}

On the other hand, from \eqref{edense} we know that
$\F^u_g(g^{-m_1n_0-m_0}(\xi))\cap
D^{cs}_g(g^{-m_1n_0}(z),\ep)\neq\emptyset.$ Using \eqref{mu},
iterating $m_1n_0+m_0$ times we conclude that
$$\F^u_g(\xi)\cap U\neq\emptyset$$
as we wished. This completes the proof of the minimality of $\F^u_g$
for $g\in \mathcal U(g_{B,k})$ with $k$ and $\mathcal U$ small
enough.
\end{proof}

\subsection{Semiconjugacy with  the linear Anosov System.}\label{secsemiup}

In this subsection we establish  a well known result about the
semiconjugacy of any map isotopic to an Anosov map on the torus (see
for instance  \cite{S}) and also we derive some consequence of it.
Indeed, we establish it in the universal cover $\R^3.$

\begin{teo}\label{semi} Let $B:\R^3\to \R^3$ be a linear hyperbolic
isomorphism. Then, there exists $C>0$ such that if $G:\R^3\to\R^3$
is homoemorphism such that $\sup\{\|G(x)-Bx\|:x\in\R^3\}=K<\infty$
then there exists $H:\R^3\to\R^3$ continuous and onto such that:
\begin{enumerate}
\item $B\circ H= H\circ G.$
\item $\|H(x)-x\|\le CK$ for all $x\in\R^3.$
\item $H(x)$ is characterized as the unique point $y$ such that
$$\|B^n(y)-G^n(x)\|\le CK \,\,\,\forall \,\,\,n\in\Z.$$
\item $H(x)=H(y)$ if and only if $\|G^n(x)-G^n(y)\|\le 2CK \,\forall\,n\in\Z$ and if and only if $\sup_{n\in\Z}\{\|G^n(x)-G^n(y)\|\}<\infty.$
\item If $B\in SL(3,\Z)$ and $G$ is the lift of $g:\T^3\to\T^3$ then
$H$ induces $h:\T^3\to\T^3$ continuous and onto such that $B\circ
h=h\circ g$ and $dist_{C^0}(h,id)\le C dist_{C^0}(B,g).$
\end{enumerate}

\end{teo}

We will prove some consequence of the above theorem to our $B\in
SL(3,\Z)$ and our construction of Ma\~{n}e's derived from Anosov
diffeomorphism $g_{B,k}:\T^3\to\T^3$ and any $g\in\mathcal
U(g_{B,k}).$ Let $G:\R^3\to \R^3$ the lift of $g$ such that
$\sup\{\|G(x)-Bx\|:x\in\R^3\}=dist_{C^0}(B,g)$ (that we may assume
that is less than $\sqrt{k}).$ Denote by $\tilde{\F}^j, j=s,c,u, cs,
cu$ the lift of the stable, central, ustable, central stable and
central unstable foliation respectively.

\begin{teo}\label{foliaciones}
With the above notations we have:
\begin{enumerate}

\item  $H(\tilde{\mathcal{F}}_G^{cu}(x))=\tilde{\F}^{cu}_B(H(x))$ and $H(\tilde{\F}_G^{cs}(x))=\tilde{\F}^{cs}_B(H(x))$
\item $H(\tilde{\F}_G^{c}(x))=\tilde{\F}^{c}_B(H(x)).$
\item $H(\tilde{\F}^u_G(x))=\tilde{\F}^u_B(H(x))=H(x)+E^u_B$ and
$H:\tilde{\F}^u_G(x)\to\tilde{\F}^u_B(H(x))$ is a homeomorphism.
\item For any $x,y\in\R^3$ hold
$$\#\{\tilde{\F}^{cs}_G(x)\cap
\tilde{\F}^u_G(y)\}=1\,\,\,\,\mbox{ and
}\,\,\,\,\#\{\tilde{\F}^{cu}_G(x)\cap \tilde{\F}^s_G(y)\}=1.$$
\end{enumerate}
\end{teo}

\begin{proof} For the first item we only prove that $H(\tilde{\F}_G^{cu}(x))=\tilde{\F}^{cu}_B(H(x))$, the other one
 is similar.
Let us prove first that
$H(\tilde{\F}_G^{cu}(x))\subset\F^{cu}_B(H(x))=H(x)+E^{cu}_B.$
Arguing by contradiction, assume that there exists $y\in
\tilde{\F}_G^{cu}(x)$ such that $H(y)\notin \tilde{\F}^{cu}_B(H(x))$
let $z=\tilde{\F}^s_B(H(y))\cap \tilde{\F}^{cu}_B(H(x)).$ By
backward iteration we have that
\begin{eqnarray*}
\|B^{-n}(H(y))-B^{-n}(H(x))\| &\ge & \|B^{-n}(H(y))-B^{-n}(H(z))\|-\|B^{-n}(H(z))-B^{-n}(H(x))\|\\
&\ge& \la_s^{-n}\|H(y)-H(z)\| -\la_c^{-n}\|H(z)-H(x)\|
\end{eqnarray*}
On the other hand,  since $y\in \tilde{\F}^{cu}_G(x)$ and (for $k$
and $\mathcal U$ small) $\|dG^{-1}_{/E^{cu}_G}\|\le
(\la_c-\ep)^{-1}$ we have
\begin{eqnarray*}
\|B^{-n}(H(x))-B^{-n}(H(y))\| &\le & \|B^{-n}(H(x))-G^{-n}(x)\| \\&
&+\, \|G^{-n}(x)-G^{-n}(y)\|\\ & & + \,\|B^{-n}(H(y))-G^{-n}(y)\|
\\&\le& 2C\sqrt{k}
+(\la_c-\ep)^{-n}dist_{\tilde{\F}^{cu}_G(x)}(x,y).
\end{eqnarray*}
For $n$ large enough we arrive to a contradiction with the previous
equation.

Now, since $\|H-Id\|\le C\sqrt{k}$ we have:
\begin{itemize}
\item $\tilde{\F}^{cu}_G(x)\subset \{z: dist_{\R^3}(z, H(x)+ E^{cu}_B)\le C\sqrt{k}\}$ (that is, roughly speaking, $\tilde{\F}^{cu}_G$ is a surface in a sandwich of size $C\sqrt{k}$ with central slice the plane $H(x)+ E^{cu}_B.$ See Figure \ref{fcu}.
    \item $\tilde{\F}^{cu}_G(x)$ is transversal to $E^s_B.$
    \item $\tilde{\F}^{cu}_G(x)$ is a complete manifold
\end{itemize}
and  it is not difficult to see that $\tilde{\F}^{cu}_G(x)$ is a
graph of a map $E^{cu}_B\to E^s_B$. Then, since $\|H-Id\|\le C\rho$
is follows that $H:\tilde{\F}^{cu}_G(x)\to \tilde{\F}^{cu}_B(H(x))$
is onto.

\begin{figure}[ht]\begin{center}

\psfrag{fcu}{$\tilde{\F}^{cu}_G(x)$} \psfrag{H1}{$H(x)-C\rho e_s+
E^{cu}_B$} \psfrag{H2}{$H(x)+C\rho e_s+ E^{cu}_B$}

\includegraphics[height=5cm]{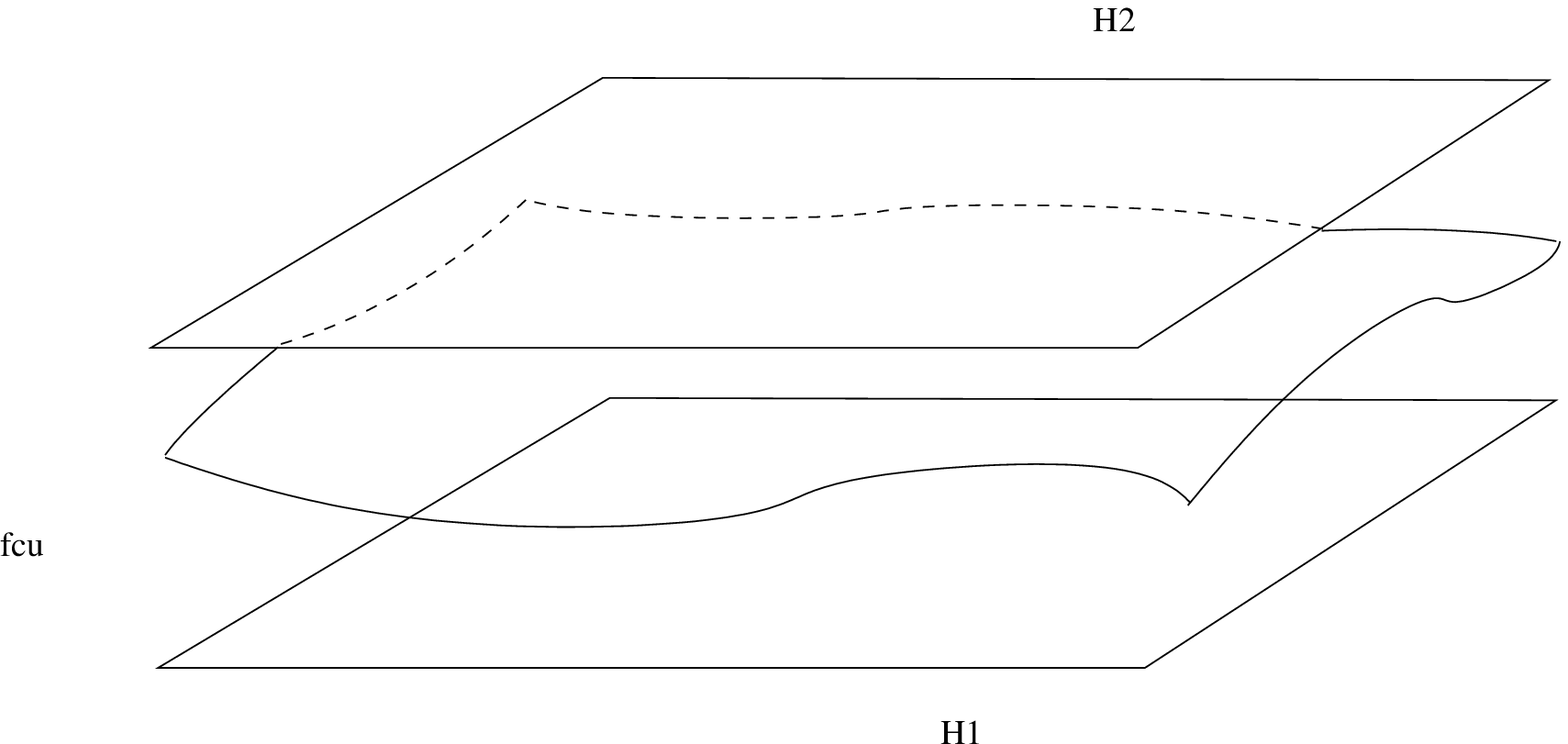}
\caption{The $\F^{cu}_G$ leaf}\label{fcu}
\end{center}\end{figure}

Let us prove the second item. From the first one it follows that:
$$H(\tilde{\F}^c_G(x))=H(\tilde{\F}^{cs}_G(x)\cap \tilde{\F}^{cu}_G(x))\subset \tilde{\F}^{cs}_B(H(x))\cap \tilde{\F}^{cu}_B(H(x))=\tilde{\F}^{c}_B(H(x)).$$
Since $\|H-Id\|\le C\sqrt{k}$ we have that $\tilde{\F}^c_G(x)$ ia in
a cylinder of radius $C\sqrt{k}$ with axe
$H(x)+E^c_B=\tilde{\F}^c_B(H(x)).$ Since $E^c_G$ is in tiny cone
around $E^c_B$ we may assume that $E^c_B$ is always transversal to
$E^s_B\oplus E^u_B$ and moreover $\tilde{\F}^c_G(x)$ is the graph of
a map $E^c_B\to E^s_B\oplus E^u_B.$ Using again that $\|H-Id\|\le
C\sqrt{k}$ we conclude that $H:\tilde{\F}^c_G(x)\to
\tilde{\F}^c_B(H(x))$ is onto.

For the third item item observe also that
$H(\tilde{\F}^u_G(x))\subset \tilde{\F}^u_B(H(x))$ since  for $y\in
\tilde{\F}^u_G(x)$ we have that $\|G^n(y)-G^n(x)\|\to_{n\to -\infty}
0$ and hence the distance between $H(G^n(y))=B^n(H(y))$ and
$H(G^n(x))=B^n(H(x))$ is bounded for $n\le 0$ which implies that
$H(y)\in H(x)+E^u_B.$ By similar arguments as in the previous item
we have that $H:\tilde{\F}^u_G(x)\to \tilde{\F}^u_B(H(x))$ is onto.
On the other hand, $H_{/\tilde{\F}^u_G(x)}$ is injective: otherwise,
if for some $z,y\in\tilde{\F}^u_G(x)$ we have that $H(z)=H(y)$ by
forward iteration we have that $\|G^n(y)-G^n(z)\|$ goes to infinity
(recall that $\tilde{\F}^u_G$ is quasi isometric) and so
$\|H(G^n(y))-H(G^n(z))\|$ also goes to infinity by forward
iteration, this is imposible since
$H(G^n(y))=B^n(H(y))=B^n(H(z))=H(G^n(z)).$

For the fourth and last item observe that
$$\#\{\tilde{\F}^{cs}_G(x)\cap
\tilde{\F}^u_G(y)\}\le 1.$$ Otherwise, let $z,w\in
\tilde{\F}^{cs}_G(x)\cap \tilde{\F}^u_G(y)$ and iterating forward we
have (since $\tilde{\F}^u$ is quasi isometric) that
$\|G^n(z)-G^n(w)\|\sim dist_{\tilde{\F}^u}(G^n(z),G^n(w))$ which
grows with exponential rate $\sim\la_u.$ On the other hand, since
$z,w\in \tilde{\F}^{cs}$ the distance can grow at most with rate
$\la_c(g)<1+\ep<\la_u$ and we get a contradiction.

To see the intersection is nonempty just recall that
$\tilde{\F}^{cs}(x)$ is a graph of a (bounded) map $E^{cs}_B\to
E^u_B$ and $\tilde{\F}^u(y)$ is a graph of a (bounded) map $E^u_B\to
E^{cs}_B.$

The second part of this item is very similar to what we already
done. Nevertheless (for the very last argument) it worth to mention
that it is not true in general that
$H(\tilde{\F}^s_G(x))=\tilde{\F}^s_B(H(x))$, and so we can not be
sure that $\tilde{\F}^s_G(x)$ is at a bounded distance of
$H(x)+E^s_B$ but still it is not difficult to see (since $E^s_G$ is
in a tiny cone around $E^s_B$) that $\tilde{\F}^s_G(x)$ is the graph
of a map $E^s_B\to E^{cu}_B.$

\end{proof}

\begin{cor}\label{puntosequivalentes}
With the above notations, assume that  $H(x)=H(y).$ Then $x,y$
belongs to the same central leaf
$\tilde{\F}^c_G(x)=\tilde{\F}^c_G(y).$ Moreover, if we denote by
$[x,y]_c$ the central arc in $\tilde{\F}^c_G(x)$ with ends $x$ and
$y$ then $H([x,y]_c)=H(x)=H(y)$ and the diameter of $[x,y]_c$ is
bounded by $2C\sqrt{k}.$ In particular, for any $z$ we have that $H^{-1}(z)$ is
either a point or an arc.
\end{cor}

\begin{proof}
Let $x,y$ be such that $H(x)=H(y)$. We claim that $y\in
\tilde{\F}^{cs}_G(x).$ Otherwise, from the last theorem we may
consider $z=\tilde{\F}^{cs}_G(x)\cap\tilde{\F}^{u}_G(y).$ By similar
arguments as before, since by forward iteration the distance between
$G^n(z)$ and $G^n(y)$ grows with a rate much higher than the one
between $G^n(z)$ and $G^n(x)$ could do, we conclude that
$$\|G^n(x)-G^n(y)\|\to_{n\to\infty}\infty.$$
This is impossible due to $H(G^n(y))=H(G^n(x))$ and so $G^n(z)$ and
$G^n(y)$ are at bounded distance for every $n.$

In a similar way we prove that $y\in \tilde{\F}^{cu}_G(x).$
Therefore
$$y\in \tilde{\F}^{cs}_G(x)\cap \tilde{\F}^{cu}_G(x)=\tilde{\F}^c_G(x).$$

Now, recall that $\tilde{\F}^c(z)$  is the graph of a map
$H(z)+E^c_B\to H(z)+E^s_B\oplus E^u_B$ and bounded by $C\sqrt{k}$
(in particular $\tilde{\F}^c(z)$ is quasi isometric) for any $z.$ We
shall denote by $\Pi^{su}:\R^3\to E^c_B$  the projection along
$E^s_B\oplus E^u_B.$

Now, if $w\in [x,y]_c$ it follows that for any $n$ that
$\Pi^{su}(G^n(x))<\Pi^{su}(G^n(w))<\Pi^{su}(G^n(y)).$ From this it
follows that $\sup_{n\in\Z}\{\|G^n(x)-G^n(y)\|\}<\infty$ and so
$H(x)=H(w).$ Finally, if $H(w)=H(z)$ then $\|z-w\|\le 2C\sqrt{k}.$

\end{proof}

Let us set the following notation: for $x\in\R^3$ let
$[x]=\{y\in\R^3: H(y)=H(x)\}=H^{-1}(H(x)).$ In other words $[x]$ is
the equivalent class or the equivalence relation $x\sim y$ if and
only if $H(x)=H(y).$ From the above lemma we have that $[x]$ is a
point or an arc contained in the central leaf $\F^c_G(x).$ In
particular from the fact $H:\tilde{F}^u_G(x)\to\tilde{F}^u_B(H(x))$
is a homeomorphism, we have (see Figure \ref{class}):

\begin{cor}\label{picclass}
Let $x\in \R^3$ and let $z\in \F^u_G(x). $ Then
\begin{equation}\label{uclass}
[z]=\left(\bigcup_{y\in[x]}\tilde{\F}^u_G(y)\right) \bigcap
\tilde{\F}^c_G(z).
\end{equation}
\end{cor}

\begin{figure}[ht]\begin{center}

\psfrag{cz}{$[z]$} \psfrag{cx}{$[x]$} \psfrag{Hx}{$H(x)$}
\psfrag{Hz}{$H(z)$} \psfrag{H}{$H$} \psfrag{x}{$x$} \psfrag{z}{$z$}
\psfrag{Fcu}{$\tilde{\F}^{cu}_G(x)$} \psfrag{Bcu}{$H(x)+ E^c_B\oplus
E^u_B$} \psfrag{Bc}{$E^c_B$} \psfrag{Bu}{$E^u_B$}
\psfrag{Fc}{$\tilde{\F}^c_G$} \psfrag{Fu}{$\tilde{\F}^u_G$}
\psfrag{Fcz}{$\tilde{\F}^c_G(z)$}
\includegraphics[height=6cm]{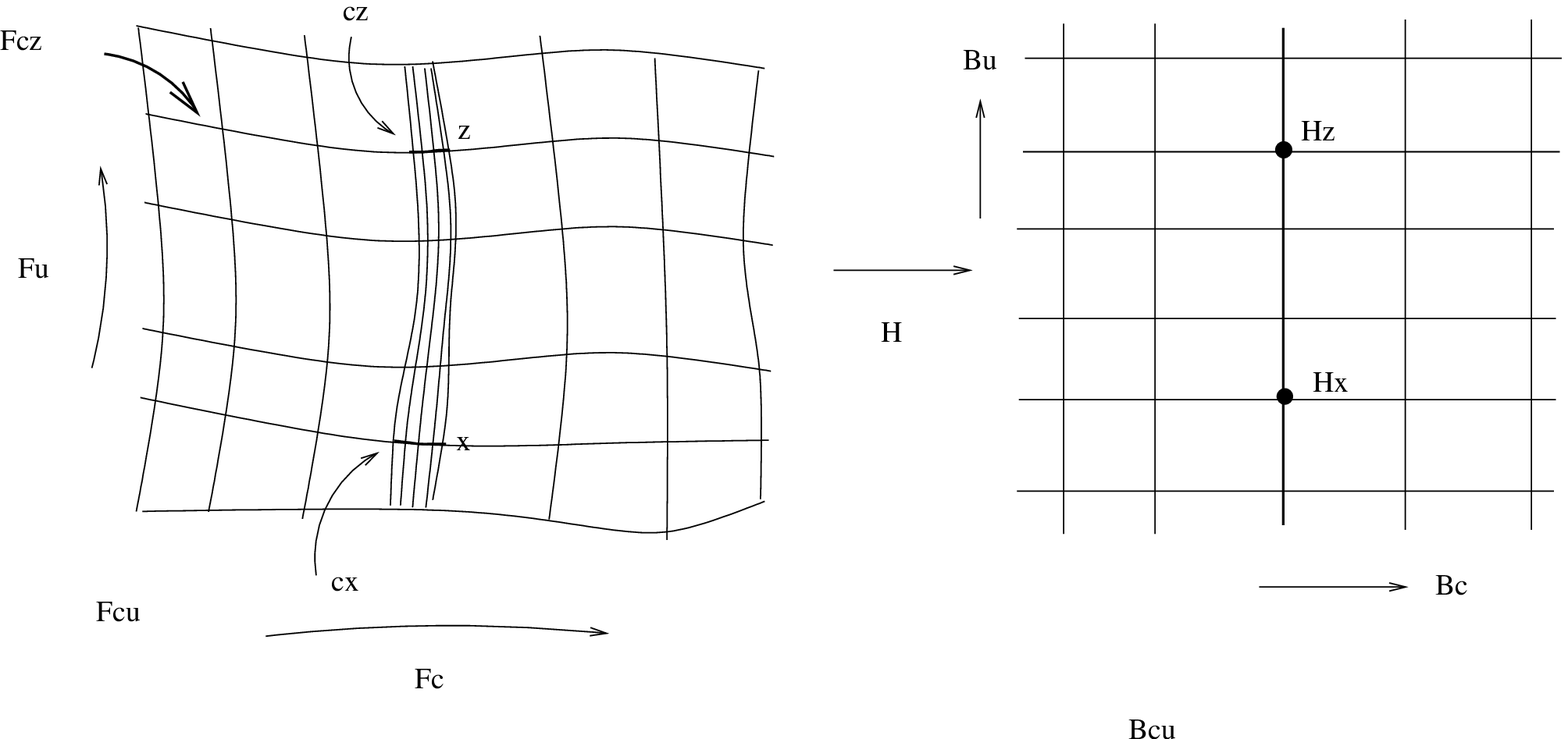}
\caption{}\label{class}
\end{center}\end{figure}

Now, going back to the induced linear anosov diffeomorphism on the
$3$-torus by $B\in SL(3,Z)$ and the Ma\~{n}\'{e}'s DA $g\in\mathcal
U(g_{B,k})$ and applying the previous results we get the following

\begin{teo}
There exists $h:\T^3\to\T^3$ continuous and onto such that
\begin{enumerate}
\item $B\circ h=h\circ g$
\item $dist_{C^0}(h,id)\le C\sqrt{k}.$
\item $h(\F^j_g(x))=\F^j_B(h(x))$ where $j=cs,cu,c, u$ and $h:\F^u_g(x)\to\F^u_B(h(x))$ is a homeomorphism.
\item If $h(x)=h(y)$ then $y\in\F^c_g(x).$
\item $h^{-1}(z)$ is either a point or an arc contained in a central leaf (with diameter less than $2C\sqrt{k}).$
\item If we set $[x]=h^{-1}(h(x))=\{y\in\T^3:h(x)=h(y)\}$ then, for $z\in\F^u_g(x)$ we have
$$[z]=\left(\bigcup_{y\in[x]}\F^u_g(y)\right) \bigcap \F^c_g(z).$$
\end{enumerate}
\end{teo}

\begin{cor}\label{cmini}
Let $g\in\mathcal U(g_{B,k})$ as above. Then, $\F^c_g$ is minimal,
i.e., every leaf is dense in $\T^3.$
\end{cor}

\begin{proof}
Let $x\in\T^3$ and let $U\subset \T^3$ be an open set. We want to
prove that $\F^c_g(x)\cap U\neq\emptyset.$ Consider $S\subset U$ a
small two dimensional disk transverse to $E^c_g.$ We know that
$h_{/S}$ is injective and hence $h(S)$ is a two dimensional
topological manifold transverse to $E^c_B.$ Since $\F^c_B$ is
minimal, we get that $\F^c_B(h(x))\cap h(S)\neq\emptyset,$ that is,
there exists $y\in S$ such that $h(y)\in \F^c_B(h(x))=h(\F^c_g(x)).$
Therefore $y\in \F^c_g(x)$ and so $\F^c_g(x)\cap U\neq\emptyset.$
\end{proof}

\subsection{Further analysis on the semiconjugacy}\label{secanalysis}
In this section we give a more detailed consequences of the
semiconjugacy with the linear Anosov diffeomorphism $B$ and on the
equivalent classes $[x]=h^{-1}(h(x))=\{y:h(y)=h(x)\}.$ Let us begin
with the following

\begin{lema}\label{full}
For $g\in\mathcal U(g_{B,k})$ as above the following hold:
\begin{enumerate}
\item If $$\liminf_{n\to-\infty}\frac{1}{n}\log\|dg^n_{/E^c_g(x)}\|>0$$ then $[x]=h^{-1}(h(x))\supsetneq\{x\}.$

\item The set $\mathcal{A}=\{z\in\T^3: h^{-1}(z)\mbox{ is a point }\}$  has full Lebesgue measure.
\end{enumerate}

\end{lema}

\begin{proof}
For the proof of the first item, let $\gamma,$
$$\liminf_{n\to-\infty}\frac{1}{n}\log\|dg^n_{/E^c_g(x)}\|>\gamma>0.$$
Then, for $n$ large enough we have
$$\|Dg^{-n}_{E^c_g(x)}\|\le e^{-\gamma}n$$ and therefore, by standard
arguments, there exists a central arc $I_c$ containing $x$ such that
the length of $g^{-n}(I_c)$ is uniformly bounded for $n\ge 0$
(indeed, $I_c\subset W^u(x)).$ We claim that $g^n(I_c)$ has bounded
length for $n\ge 0.$ We will denote by $\ell(I)$ the length of $I$.

We may assume that $\rho$ is small (recall that the support of the
modification of $B$ is in $B(p,\rho)$) so that if $J_c$ is a central
arc such that $4\rho\le \ell(J_c)\le 6\rho$ then $J_c\cap B(p,\rho)$
has at most one connected component of length at most $2\rho.$
Recall also that $\la_c(g)<1+\ep$ where $\ep$ is small (taking $k$
small) (for instance, $\ep<1-\la_c$ and $\ep<1/2.$)

To prove the claim we may assume that $\ell(I_c)<2\rho$ and arguing
by contradiction, consider the case where the length of $g^n(I_c)$
is unbounded for $n\ge 0.$ Let $n_0$ be the first time such that
$\ell(g^n(I_c)\ge 6\rho.$ Since $4\rho \la_c(g)<4\rho (1+\ep)<6\rho$
it follows that
$$4\rho\le \ell(g^{n_0-1}(I_c))< 6\rho.$$
Set $J_c=g^{n_0-1}(I_c).$ By the above condition on $J_c$ and
recalling that $\|dg_\xi\|=\|B\|=\la_c$ if $\xi\notin B(p,\rho)$ we
get
$$6\rho\le \ell(g^{n_0}(I_c))=\ell(g(J_c))\le (1+\ep)\frac{\ell(J_c)}{2} +
\la_c\frac{\ell(J_c)}{2}<\ell(J_c)<6\rho,$$ a contradiction. Now
since, $\ell(g^n(I_c))$ is bounded for all $n\in\Z$ we conclude that
$h(I_c)=h(x)$ (this can be seen by lifting to $\R^3$ where
inmediately follows that $\|G^n(x)-G^n(y)\|$ is bounded for all
$n\in\Z$ and $y\in I_c.$)

For the proof of the second item, we may assume that
$\la_c(1+\ep)\la_c(g)<1$ and also that $m(B(p,4\rho))<\frac{1}{2},$
where $m$ is the lebesgue measure in $\T^3.$ Since $B:\T^3\to\T^3$
preserves measure and it is ergodic there is a full measure set
$\mathcal{R}$ such that if $y\in\mathcal{R}$ we have
$$\lim_{n\to\infty}\frac{\#\{j:0\le j\le n, B^{-j}(y)\in
B(p,4\rho)\}}{n}=m(B(p,4\rho))<\frac{1}{2}.$$

We will show that $\mathcal{R}\subset\mathcal{A}.$ Arguing by
contradiction, let $y\in\mathcal{R}$ such that $h^{-1}(y)$ is a
nontrivial center arc $I_c.$ Recall that $\ell(g^n(I_c))$ is bounded
(by $2C\sqrt{k}<\rho$) for all $n\in\Z.$ Thus, whenever we have that
$B^{-j}(y)\notin B(p,4\rho)$ then $g^{-j}(I_c)\cap
B(p,\rho)=\emptyset.$ Since $dg_{/E^c_g(\xi)}\le \la_c(1+\ep)$ for
$\xi\notin B(p,\rho)$, if $J_c\cap B(p,\rho)=\emptyset$ then
$\ell(g^{-1}(J_c))\ge \left(\la_c(1+\ep)\right)^{-1}\ell(J_c).$ And
in any case $\ell(g^{-1}(J_c))\ge \la_c(g)^{-1}\ell(J_c)$

Now, for $n$ large enough we have:

\begin{eqnarray*}
\ell(g^{-n}(I_c)) &\ge &\left( \prod_{j: B^{-j}\notin
B(p,4\rho)}\left(\la_c(1+\ep)\right)^{-1}\prod_{j: B^{-j}\in
B(p,4\rho)}\la_c(g)^{-1}\right)
\ell(I_c)\\
&\ge&
(\la_c(1+\ep)\la_c(g))^{-\frac{n}{2}}\ell(I_c)\to_{n\to\infty}\infty,
\end{eqnarray*}
a contradiction.
\end{proof}

The following lemma says that in any unstable leaf there is a point
whose forward orbit never meets $B(p,2\rho)$ and is similar to what
we have done in Section \ref{secmini}. Notice also that $\F^u_g$ is
orientable and choose an orientation. For $x\in \T^3$ denote by
$\F^{u,+}_g(x,t)$ an arc of length $t$ in $\F^u_g(x)$ starting at
$x$ in the chosen orientation.

\begin{lema}\label{ufuera}
Assume that $\la_u>3.$ Then for $\rho, k$ and $\mathcal{U}$ small the
following holds for $g\in\mathcal U(g_{B,k}):$ for any $x\in\T^3$
there exists a point $z_x\in \F^{u,+}_g(x,1)$ such that $g^n(z)\cap
B(p,2\rho)=\emptyset$ for any $n\ge 0.$
\end{lema}
\begin{proof}
We may assume that $\rho$ is so small that any segment $I_u$ in
$\F^u_B$  of length one has a subsegment $I_{u_1}$ of length $1/3$
such that $I_{u_1}\cap B(p,2\rho)=\emptyset.$ Now, if $k$ is small and
$\mathcal{U}(g_{B,k})$ as well we may assume that the same property
holds for $g\in\mathcal U,$ that is, any arc $I_u$ in $\F^u_g$  of
length one has a subarc $I_{u_1}$ of length $1/3$ such that $I_{u_1}\cap
B(p,2\rho)=\emptyset.$ Moreover, we may assume that $\la^u(g)>3.$
Now, $g(I_{u_1})$ has length at leat one and so it has a subarc
$I_{u_2}$ such that $I_{u_2}\cap B(p,2\rho)=\emptyset.$ By induction,
for any $n$ we have that $g(I_{u_n})$ contains $I_{u_{n+1}}$ such that
$I_{u_{n+1}}\cap B(p,2\rho)=\emptyset.$ Therefore,
$$z_x\in \bigcap_{n\ge 0} g^{-n}(I_{u_{n+1}})$$
satisfies the lemma.
\end{proof}

\begin{cor}\label{classto0}
Let $g\in\mathcal U$ as above and let $x\in\T^3$ be such that
$[x]\supsetneq \{x\}.$ Then, given ay $\eta>0$ there is a point $y\in
\F^{u,+}_g(x)$ (the positive side of $\F^u_g$ in the chosen
orientation) such that the length $\ell([y])<\eta.$
\end{cor}

\begin{proof}
Recall that for $g\in\mathcal{U}$ we have
$\|dg_{/E^c_g(\xi)}\|<\la_c(1+\ep)<1$ if $\xi\notin B(p,\rho).$
Also, if $k$ is small then $2C\sqrt{k}<\rho.$ Let $\eta$ be given
and let $n_0$ be such that
$$\left(\la_c(1+\ep)\right)^{n_0}2C\sqrt{k}<\eta.$$
Consider $x$ such that $[x]\supsetneq \{x\}.$ From the above lemma,
consider $z\in \F^u_g(g^{-n_0}(x),1)$ such that $g^n(z)\notin
B(p,2\rho)$ for any $n\ge 0.$ Notice that, since $[g^{-n_0}(x)]$ is
not trivial, the same is true for $z.$ On the other hand $[z]$ is a
central segment of length at most $2C\sqrt{k}.$ Therfore,
$g^n([z])\cap B(p,\rho)=\emptyset$ for $n\ge 0.$ Therefore,
$$\ell(g^n[z])\le (\la_c(1+\ep))^n2C\sqrt{k}.$$
Finally, setting $y=g^{n_0}(z)\in \F^{u,+}_g(x)$ we have
$$\ell([y])=\ell(g^{n_0}[z])\le (\la_c(1+\ep))^{n_0}2C\sqrt{k}<\eta.$$
\end{proof}

The next result is fundamental for our purpose on the  behavior of
the holonomy map along the unstable foliation. The main tool is the
existence of a transversal homoclinic point (recall Corollary
\ref{homo}).

\begin{lema}\label{pclass}
Let $g\in \mathcal U(g_{B,k})$ having a transversal homoclinic point
associated to the fixed point  $p$ of unstable index two. There
exists $\ep_0$  and $z_p\in \F^{u,+}_g(p)$ such that
$$\limsup_{n\to\infty}\ell(g^n([z_p]))>\ep_0.$$
\end{lema}

\begin{proof}
Recall that $[p]$ is the central segment between $q_1,q_2.$ Let
$\ep_0<\min\{\ell[q_1,p]^c, \ell[p,q_2]^c\}.$ Notice also that
$$W^u(p)=\bigcup_{y\in (q_1,q_2)^c}\F^u_g(y).$$

Let $z$ be a homoclinic point associated to $p,$ that is $z\in
\F^s_g(p)\cap W^u(p).$ We know that
$$[z]=\left(\bigcup_{y\in[p]}\F^u_g(y)\right)\bigcap \F^c_g(z).$$
We may assume that the orientation in $\F^u_g$ is such that
$z_p=[z]\cap \F^u_g(p)\in\F^{u.+}_g(p).$ Since $[z]=[z_p], z\in
\F^s(p)$ and $[z]\subset \F^c_g(z)\subset \F^{cs}_g(p)$ by forward
iteration $g^n([z])$ must approach at leat to $[q_1,p]$ or $[p,q_1]$
(see also Figure \ref{zp}), and the lemma follows.
\end{proof}

\begin{figure}[ht]\begin{center}

\psfrag{[z]}{$[z]$} \psfrag{cx}{$[x]$} \psfrag{Fsp}{$\F^s_g(p)$}
\psfrag{I}{$[z]$} \psfrag{p}{$p$} \psfrag{q1}{$q_1$}
\psfrag{q2}{$q_2$} \psfrag{x}{$x$} \psfrag{z}{$z$}
\psfrag{Fcp}{${\F}^{c}_g(p)$}
\psfrag{Fup}{${\F}^{u}_g(p)$}
 \psfrag{gnz}{$g^n(z)$}
\psfrag{cgnz}{$g^n([z])$} \psfrag{Wux}{$W^u_\delta(x)$}

\includegraphics[height=8cm]{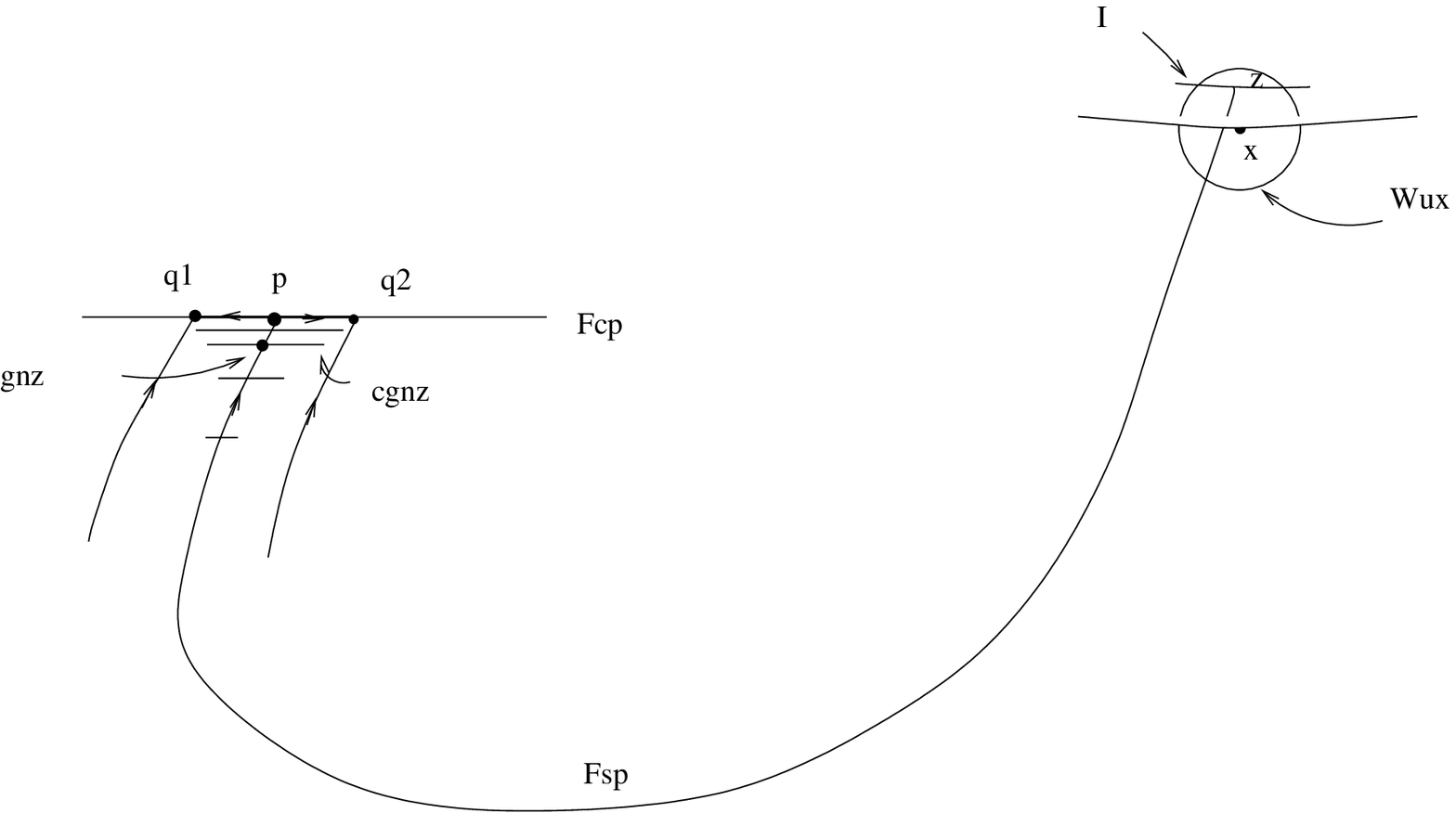}
\caption{}\label{zp}
\end{center}\end{figure}

Indeed, a more extensive result holds:

\begin{prop}\label{chaos}
Let $g\in \mathcal U(g_{B,k})$ having a transversal homoclinic point
associated to the fixed point  $=p$ of unstable index two. Then
there exists an uncountable set $\Lambda_0$ such that:
\begin{enumerate}
\item If $x,y\in \Lambda_0, x\neq y$ then $\F^{cu}_g(x)\neq \F^{cu}_g(y).$
\item For any $x\in \Lambda_0, [x]$ is nontrival.
\item There exists $\ep_0$ such that for any $x\in \Lambda_0$ and any $t>0$ there exists $z_x\in \F^{u,+}_g(x)\setminus \F^u_g(x,t)$ such that $\ell([z_x])>\ep_0.$
\end{enumerate}
\end{prop}

\begin{proof}
From the existence of a transversal homoclinic point associated to
$p$ of index two we conclude the existence of a non trivial
hyperbolic compact invariant set $\Lambda$ (of unstable index $2$)
and with local product structure. In particular, from Lemma
\ref{full} we get that for any $x\in \Lambda, [x]$ is nontrival.

Notice that for  $x\in\Lambda, W^u(x)$ is two dimensional and
contained in $\F^{cu}_g(x)$ and there exists some $\delta>0$ such
that $W^u_\delta(x)$ has uniform size.  We will denote by
$W^{u,+}_\delta(x)$ the component of $W^u_\delta(x)\setminus
\F^c_g(x)$ which is in the positive direction of $\F^u_g(x).$
Moreover, there is an uncountable number of disjoints unstable
manifolds $W^u.$ Furhtermore, there is some $L$ such that (setting
$\F^s_g(p,L)=W^s_L(p)$) we have that $$\F^s_g(p,L) \cap
W^u_\delta(x)\neq\emptyset\,\,\,\,\forall\,\,\,x\in\Lambda.$$
Indeed, it is not difficult to see that if $x$ is not in a central
stable periodic leaf, then
$$\F^s_g(p,L) \cap W^{u,+}_\delta(x)\neq\emptyset.$$
Let us see a consequence of the above fact. Let $z\in\F^s_g(p,L)
\cap W^{u,+}_\delta(x)$ and let $\ep_0<\ell([p])/2.$ Since
$g^n(z)\to_n p$ we conclude that $\ell(g^n[z])=
\ell([g^n(z)])>\ep_0$ for $n$ large enough (see Figure \ref{zp}).
Indeed, $[z]$ is a central arc of uniforme size and therefore, and
since exists $m_0$ such that $g^{m_0}(z)\in W^s_{loc}(p)$ we have
that $g^{m_0}([z])$ is central arc of uniform size in
$\F^{cs}_{loc}(p)$. Now, by forward iteration, we have that
$\ell(g^n([z]))>\ep_0$ for all $n\ge m_1$ for some $m_1$ (which is
independent of $x$).

Now choose an uncountable set $\Lambda_0\subset \Lambda$ such that
for $x\neq y\in \Lambda_0$ we have $\F^{cu}_g(x)\neq\F^{cu}_g(y)$
and that for any $x\in \Lambda_0$ then $x$ is not in a periodic
central stable leaf. It remains to prove (3). Let $x\in\Lambda_0$
and let $t>0$, and choose $n_1$ bigger than $m_1$ such that
$g^{-n_1}(\F^{u,+}_g(x,t))\subset W^u_\eta(g^{-n_1}(x))$ where
$\eta$ is such that $W^u_\eta(g^{-n_1}(x))\cap
\F^s_g(x,L)=\emptyset.$ Let $w\in \F^s_g(p,L) \cap
W^{u,+}_\delta(g^{-n_1}(x))).$ It follows that $[w]\cap
\F^{u,+}_g(g^{-n_1}(x)))\neq\emptyset$ and set $y$ the point of
intersection. Notice that in one hand $y\notin g^{-n_1}(\F^{u,+}_g(x,t))$ and
therefore $z_x=g^{n_1}(y)\in \F^{u,+}_g(x)\setminus\F^{u,+}_g(x,t).$
On the other hand
$$\ell([z_x])=\ell([g^{n_1}(y)])=\ell(g^{n_1}([y]))=\ell(g^{n_1}([z]))>\ep_0.$$

\end{proof}

Finally, we will give a result for $g$ as in Corollary \ref{homo}, which is fundamental in order to get Li-Yorke chaos:

\begin{prop}\label{lychaos}
Let $g$ as in Corollary \ref{homo} and let $p_1,p_2$ any two distinct points in $J$. Then, there exists $w\in J$ between $p_1$ and $p_2$, $z\in\F^{u,+}(w)$ and a nontrivial arc $I_c\subset[z]$ such that
\begin{enumerate}
\item $I_c\subset\left(\bigcup_{y\in[p_1,p_2]}\F^u_g(y)\right)\bigcap [z].$
\item $\lim_{k\to\infty}\ell(g^{km}(I_c))>0$ (where $m$ is given in the corrolary \ref{homo}). 
\item $g^{km}(I_c)\subset \left(\bigcup_{y\in[p_1,p_2]}\F^u_g(y)\right)$ for all $k$
\end{enumerate}
\end{prop}

\begin{figure}[ht]\begin{center}

\psfrag{Ic}{$I^c$} \psfrag{p1}{$p_1$}
\psfrag{p2}{$p_2$}  \psfrag{z}{$z$}

\psfrag{Fuw}{${\F}^{u}_g(w)$}
 \psfrag{J}{$J$}
\psfrag{D}{$D^{cs}$} \psfrag{w}{$w$}

\includegraphics[height=10cm]{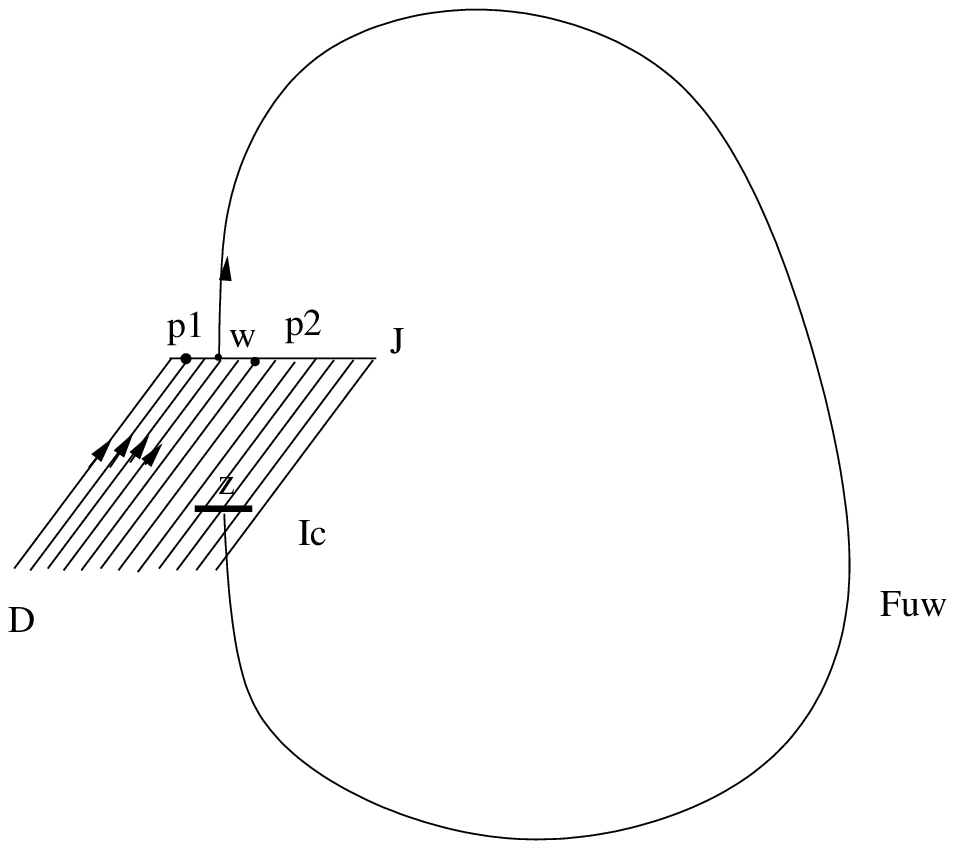}
\caption{}\label{ly}
\end{center}\end{figure}

\begin{proof}
Consider $D^{cs}=\bigcup_{x\in int(J)}\F^s_g(x)$ which contains a disc in a central stable manifold. Now, let $w\in J$ between $p_1$ and $p_2.$ Since the unstable foliation is minimal we have that $\F^u_g(w)\cap D^{cs}\neq\emptyset.$ Let $z$ be point in this intersection. Then, since $\bigcup_{y\in[p_1,p_2]}\F^u_g(y)\cap[z]$ contains $z$ in its interior (respect to $[z]$) we may find $I^c$ satisfying item (1) of the proposition and such that $I^c\subset D^{cs}.$ Since $g^m_{/J}\equiv Id$ we conclude also item (2) of the proposition. Since all points of $J$ are fixed by $g^m$, the unstable manifolds of points of $J$ are invariant by $g^m$ and so we conclude also (3).

\end{proof}

\section{The induced holonomy map on $\T^2$.}\label{secholonomy}

Let  $B\in SL(3,\Z)$ (with eigenvalues $0<\la_s<\la_c<1<\la_u$) and
$g_{B,k}$ defined in \eqref{id} and \eqref{g} and let $g\in\mathcal
U(g_{B,k})$ with $k$ and $\mathcal U$ small, in the conditions of Corollary \ref{homo}. In particular all results of the last section hold.

Consider a bidimensional torus transversal to $\F^u_B$ and (assuming
$k$ and $\mathcal U$ small) also transversal to $\F^u_g.$ For
instance, we may consider
$\T^2=\R^2_{/\Z^2}\subset\R^3_{/\Z^3}=\T^3.$

The foliations $\F^u_B$ and $\F^u_g$ are orientable and choose
similar orientation on both (that is, take unit vector fields
$X_B=e^u$ and $X_g$ close to $X_B).$

\begin{defi}\label{induced}
For $g$ as above we define $f=f_g:\T^2\to\T^2$ the holonomy map on
$\T^2$ induced by the unstable foliation $\F^u_g.$ In other words,
$f(x)$ is the first return map of $\F^u_g(x)$ to $\T^2$ in the given
orientation. Moreover, we can define $F:\T^3\to\T^2$ as the first
return to $\T^2$ of any $x\in\T^3$ along the positive orientation of
$\F^u_g(x).$
\end{defi}

\begin{obs}\label{dif}
Notice that the induced map $f=f_g$ is a homeomorphism. Morevoer,
$f$ is of class $C^r$ if the unstable foliation $\F^u_g$ is of class
$C^r$. Furthermore, the unstable foliation $\F^u_g$ is of class
$C^r$ if the unstable bundle $E^u_g$ is of class $C^r.$
\end{obs}
 Besides, if
we consider the holonomy map $T_B:\T^2\to\T^2$ induced by $\F^u_B$
we obtain that $T_B$ is a minimal (and hence ergodic) translation.
Moreover, $f=f_g$ and $T_B$ are close as we wish if $k$ is small.

If we apply the results on the previous section we obtain the
topological version of our main result:

\begin{teo}\label{maintopologico}
For $g:\T^3\to\T^3$ as above and $f=f_g:\T^2\to\T^2$ and
$T_B:\T^2\to\T^2$ as above we have:
\begin{enumerate}\renewcommand{\theenumi}{\roman{enumi}}
\item\label{mini} $f$ is minimal.

\item\label{iso} $f$ is isotopic and semiconjugated to the  ergodic translation
$T_B.$ If we denote by $\h$ the semiconjugacy, then $\h^{-1}(x)$ is
either a point or an arc.  Moreover, there are uncountable points
$x$ such that $\h^{-1}(x)$ is a nontrivial arc.

\item\label{lami} $f$ preserves a minimal and  invariant $C^0$ foliation with one dimensional $C^1$  leaves. The fibers $\h^{-1}(x)$ are contained in the leaves of this foliation.

\item\label{entro} $f$ has zero entropy.

\item\label{dis} $f$ is point-distal non-distal.

\item\label{liyorke} $f$ exhibits Li-York chaos.
\item\label{sens} $f$ has sensitivity with respect to initial conditions.

\item\label{ue} $f$ is uniquely ergodic.
\end{enumerate}
\end{teo}

\begin{proof}
(\ref{mini}) follows from the minimality of the unstable foliation
$\F^u_g$ (see section \ref{secmini}).

Let's prove \eqref{iso}. Since $f$ and $T_B$ are $C^0$ close, we get
that $f$ and $T_B$ are isotopic. Recall $h$ to be the semiconjugacy
between $g=g_{B,k}$ and $B:\T^{3}\to \T^{3}$ given in section
\ref{secsemiup}.

Since $dist_{C^0}(h,id)<C\sqrt{k}$ (which we may assume less than
$1/4$), for every point in $h(\T^{2})$ we can define a natural
projection $P:h(\T^{2})\to\T^{2}$ along the unstable foliation
$\F^u_B$, that is $P(h(x))$ is the closest within $\F^u_B(h(x))$ in
$\T^2$ to $h(x).$ Define
$$\h:\T^2\to\T^2 \,\,\,\,\,\,\,\,\h(x)=P(h(x)).$$
Clearly, $\h$ is continuous and close to the identity (if $k$ is
small) and hence onto (and isotopic to the identity as well).

Now, if we take $x\in\T^{2}$ and $f(x)\in \T^{2}$ we have that they
are the ends of an arc $I^u\subset \F^u_g(x)$  and when lifted to
$\R^3$ their coordinates have $z$-difference equal to one.

On the other hand $h(I^u)$ is an arc (segment) of $\F^u_B(h(x))$ so
that, when lifted to $\R^3$ the ends have coordinates whose
$z$-difference is between $1-2C\sqrt{k}$ and $1+2C\sqrt{k}.$
Therefore $P(h(f(x))=T_B(\h(x))$ that is $$\h\circ f= T_B\circ \h.$$

Notice that
\begin{itemize}

\item If $h^{-1}(x)=\{y\}$ then  clearly holds that
$\h^{-1}(x)$ is a unique point.

\item If $h^{-1}(x)$ is a non trivial central arc, then the projection (by $P$) on $\T^{2}$
is a non trivial arc and it is  $\h^{-1}(x)$.

\end{itemize}
Moreover, by Proposition \ref{chaos}, we get that there are an
uncountable  number of points $x$ such that $\h^{-1}(x)$ is a
nontrivial arc. This finishes the proof of  (\ref{iso}).

To prove (\ref{lami}), for $x\in\T^2$ let $\C(x)$ the connected
component that contains $x$ of $\F^{cu}_g(x)\cap \T^2.$ It follows
that $\C$ is a continuous foliations with $C^1$ dimensional leaves
(recall that $\F^{cu}_g(x)$ is a $C^1$ manifold) and obviously
invariant by $f$, the holonomy map. Furhtermore, since
$h(\F_g^{cu})(x)=\F^{cu}_B(h(x))$ it follows that $\h(\C(x))$ is the
connected component of $\F^{cu}_B(\h(x))\cap\T^2$ which contains
$\h(x).$ Since this foliation by lines on $\T^2$ is minimal we also
conclude that $\C$ is minimal (similar proof as in Corollary
\ref{cmini}). Since $h^{-1}(x)$ live in a central unstable leaf, we
get that $\h^{-1}$ live in the leaves of this foliations.

The proof of \eqref{entro} is rather easy. Indeed, by Bowen's
formula (\cite{Bo}) we have
$$h_{top}(f)\le h_{top}(T_B)+\sup_{x\in\T^2}h_{top}(f,\h^{-1}(x))$$
where $h_{top}(f,K)=\lim_{\ep\to
0}\limsup_{n\to\infty}\frac{1}{n}\log N(\ep,n,f,K)$ and
$N(\ep,n,f,K)$ is the minimum cardinality of $(n,\ep)$ separated set
in $K.$ Since for all $x$, $\h^{-1}(x)$ is either a point or an arc
(with bounded length in the future and in the past) we have the
result (see also \cite{BSFV}).

Let us prove (\ref{dis}). Recall that $f$ is point distal if there
exists $x\in\T^2$ such that for every $y\neq x$ there exists $r_y>0$
so that $r_y\le\inf\{dist(f^n(x),f^n(y)):n\in\Z\}$ and $f$ is non
distal if there exists a pair of points $z,w$ such that
$\inf\{dist(f^n(z),f^n(w)):n\in\Z\}=0.$ We will show first that $f$
is point distal. Let $x\in\T^2$ be such that $\h^{-1}(\h(x))=\{x\}$
and consider any $y\in\T^2.$  Let $\alpha=dist(\h(x),\h(y))$. By the
(uniform) continuity of $\h$, there exists $r$ such that if
$dist(z,w)<r$ then $dist(\h(z),\h(w))<\alpha$ for any $z,w\in\T^2.$
We claim that $\inf\{dist(f^n(x),f^n(y)):n\in\Z\}\ge r>0.$
Otherwise, if for some $n$ we have $dist(f^n(x),f^n(y))<r$ then we
get (since $T_B$ is an isometry):
$$\alpha>\,dist(\h(f^n(x)),\h(f^n(y)))\,=\,dist(T_B^n(\h(x)),T_B(\h(y)))=dist(\h(x),\h(y))=\alpha.$$

Now, we will prove that $f$ is non-distal. We will go back to
$g=:\T^3\to\T^3$  and let $x$ such that
$[x]=h^{-1}(h(x)\supsetneq\{x\}.$ Let $I_x=F([x])$ the first return
map to $\T^2$ of $[x]$ along the unstable foliation $\F^u_g$ if
$[x]\cap \T^2=\emptyset,$ otherwise, set $I_x=P([x]).$  From
Corollary \ref{classto0} we know that for any $\eta$ there exists
$y\in\F^{u,+}_g(x)$ such that $\ell([y])<\eta.$ On the other hand,
since  $\displaystyle{[y]=\bigcup_{z\in[x]}\F^{u,+}_g(z)\bigcap
\F^c_g(y)}$ we have that there exists $n_y$ such that
$f^{n_y}(I_x)=I_y$. It follows that
$\liminf_{n\to\infty}(f^n(I_x)))=0.$ Finally, if we take $z\neq w\in
I_x$ we conclude that $\inf\{dist(f^n(z),f^n(w)):n\in\Z\}=0,$ i.e.,
$f$ is non-distal.

We prove now \eqref{liyorke}. Consider $J$ the arc as in Corollary \ref{homo}. Notice that for any point $x\in J, [x]\supseteq J.$ Let $\tilde{J}=F(J)$ the first return
map to $\T^2$ of $J$ along the unstable foliation $\F^u_g$ if
$J\cap \T^2=\emptyset,$ otherwise, set $\tilde{J}=P(J).$ As above, $\liminf_{n\to\infty}f^n(\tilde{J})=0.$ On the other hand, given any two points $p_1,p_2$ in $\tilde{J}$ and applying Proposition \ref{lychaos} we conclude that $\limsup_{n\to\infty}d(f^n(p_1),f^n(p_2))>0$.

For the proof of \eqref{sens}, recall that $f$ has sensitivity with
respect to initial conditions if there exists some $\ep_2$ such that
for any $x\in\T^2$ and any open set $U$ containing $x$ there exist
$y\in U$ and $n>0$ such that $dist(f^n(x),f^n(y))\ge \ep_2.$ So,
given $\ep_1$ let $\ep_2$ be such that any arc in $\mathcal C$ of
length $\ep_1$ then the endpoints are at distance at least $2\ep_2.$
Let $x$ and $U$ be given. Assume first that $\h^{-1}(\h(x))=\{x\}$
which is the same as $[x]=\{x\}.$ Since $f$ is minimal we have that
there is $m_k$ such that $f^{m_k}(p)\to_k x.$ We claim that for $k$
large enough $f^{m_k}(I_p)\subset U.$ Indeed, it follows that
$\ell(f^{m_k})\to 0,$ otherwise we conclude that $[x]\neq \{x\}$
(the equivalent classes are lower semicontinuous). Thus, choose some
$m$ so that $f^m(I_p)\subset U.$ Since $\limsup \ell(f^n(I_p))\ge
\ep_1$ we get the result taking $y$ as the appropriate end point of
$f^m(I_p).$ Now, if $[x]$ is non trivial we can argue as before,
since in $U$ there are points $z$ such that $[z]$ si trivial and so
for some $m$ we have that $f^m(I_p)\subset U.$

It is left to prove \eqref{ue}. Consider the set
$$\tilde{\mathcal{A}}=\{x\in\T^2:\h^{-1}(x)\mbox{ is a point}\}.$$
Observe that $\h^{-1}(x)$ is a point if and only if $h^{-1}(x)$ is a
point. Moreover, if $h^{-1}(x)$ is just a point the same is true for
any $y\in\F^u_B(x).$ Therefore, since
$$\mathcal{A}=\{x\in\T^3:h^{-1}(x)\mbox{ is a point }\}$$
has full lebesgue measure on $\T^3$ by Lemma \ref{full} we get that
$\tilde{\mathcal{A}}$ has full lebesgue measure on $\T^2.$

Denote  by $\mathcal{M}(f)$ the set of invariant probabilities of
$f$. Given $\mu\in \mathcal{M}_f$ we may define a measure $\nu\in
\mathcal{M}_{T_B}$ by $\nu(A)=\mu(h^{-1}(A))$. Since $T_B$ is
uniquely ergodic, $\nu=m$ (the lebesgue measure  on $\T^2$). That
is, for every borelean set $D$ and $\mu\in \mathcal{M}_f$ we have
$\mu(h^{-1}(D))=m(D)$. And therefore, for every $\mu\in
\mathcal{M}_f$, setting $\mathcal{D}=\h^{-1}(\tilde{\mathcal{A}})$
we have
$$\mu(\mathcal{D})=\mu(\h^{-1}(\tilde{\mathcal{A}}))=m(\tilde{\mathcal{A}})=1.$$
Observe that for any Borel set $A$ we have
$A\cap\mathcal{D}=\h^{-1}(\h(A\cap\mathcal{D}).$

Given $\mu_1,\mu_2\in \mathcal{M}_f$ and $A$ any Borel set we have
\begin{eqnarray*}
\mu_1(A)&=&\mu_1(A\cap \mathcal{D})=\mu_1(\h^{-1}(\h(A\cap
\mathcal{D})))=m(\h(A\cap \mathcal{D}))\\&=&\mu_2(\h^{-1}(h(A\cap
\mathcal{D})))=\mu_2(A\cap \mathcal{D})\\&=&\mu_2(A).
\end{eqnarray*}
Thus $f$ is uniquely ergodic.

\end{proof}

\begin{obs}
If $f$ were of class $C^2$ and the leaves of the foliation $\mathcal
C$ also were of class $C^2$ one is tempted to use Schwarz's argument
(\cite{Sch}) to show that non trivial fibers of $\h^{-1}$ are  not
possilbe. However, in our case there is extra difficulty: we don´t
know \textit{a priori} that the sum of the length of the iterates of
a nontrivial fiber (if exists) does converge. In our examples, this
sum does not converge!

Let us point out as well that with our method, the differentiability
of the sysetm and of the foliation are like the dishes on a balance.
More differentiability ensured for the system implies less ensured
for the folliation.
\end{obs}

\section{On the smoothness of $E^u_g$.}\label{secdif}

From Theorem \ref{maintopologico} and Remark \ref{dif} the only
thing that is left to prove for the proof of our Main Theorem is the
following: given $r\in [1,3)$ there exists $g$ so that the unstable
bundle  $E^u_g$ is of class $C^r.$

In order to establish the differentiability class of $E^u_g$ we
recall a classical result from \cite{HPS} that is very useful for
these type of problems.

\begin{teo}{$C^{r}$-section theorem.}\label{crsection}

Consider $M$ a compact $C^{r}$-manifold and $g:M\rightarrow M$ a
$C^{r}$-diffeomorphism. Let $\pi:L\rightarrow M$ be a
finite-dimensional Finslered vector bundle and let $D$ be disc
subbundle, $\pi(D)=M.$  Let $F: D \rightarrow D$ be a homeomorphisms
such that $F(L_{\xi})=L_{g(\xi)}$, and let $l_{\xi}=l_{\xi}(F,g)$ be
the Lipchitz constant of $F\mid_{L_{\xi}}$ for $\xi \in M$.

Then if $l_{\xi}<1$ for every $\xi\in M$ there exists a unique
continue section $\sigma: M \rightarrow L$ such that $F \circ \sigma
= \sigma \circ g$ (an invariant section).

Moreover, if $\pi:L\rightarrow M$ is a $C^{r}$-vector bundle (with
some structure which is compatible with the Finslered structure),
$F$ is $C^{r}$ and setting $\tau_{\xi}=\tau_{\xi}(g)=\|
(dg_{\xi})^{-1}\|$ we have $l_{\xi}.\tau_{\xi}^{r}<1$, then the
invariant section $\sigma:M\rightarrow L$ is $C^{r}$.

\end{teo}

Let $B\in Sl(3,\Z)$ be a linear transformation with eigenvalues
$0<\la_s<\la_c<1<\la_u$ and invariant hyperbolic structure
$E^s_B\oplus E^c_B\oplus E^u_B$ as we have been considering and the
euclidean metric on $\R^3$ such that the above spaces are mutually
orthogonal. Consider the vector space
$$\mathcal{L}(E^u_B,E^s_B\oplus E^c_B)=\{t:E^u_B\to E^s_B\oplus E^c_B, t
\mbox{ linear}\}$$ endowed with the natural norm structure.

Consider the (trivial) vector bundle
\begin{equation}\label{fibrado}
L=\{(\xi,t):\xi\in\T^{3},t\in \mathcal{L}(E^u_B,E^s_B\oplus
E^c_B)\}.
\end{equation}
Then $\pi:L\rightarrow M$ given by $\pi(\xi,t)=\xi$ is a (finite
dimensional) $C^\infty$ Finslered vector bundle.

Now, for $g=g_{B,k}:\T^3\to\T^3$ we define the associated vector
bundle map  $F=F_{B,g}:L\rightarrow L$ as follows: for $(\xi,t)\in
L,$
\begin{equation}\label{grafico}
F(\xi,t)=(g(\xi),s)),s\in \mathcal{L}(E^u_B,E^s_B\oplus E^c_B) \mbox{
such that } graph(s))=dg_\xi(graph(t)).
\end{equation}
   Recall that $E^s_B\oplus
E^c_B$ is invariant by $dg_\xi$ for any $\xi\in\T^3$ and so $F$ is
well defined vector bundle homeomorphism. Neverhteless, for $g$
close to $g_{B,k}$ the associated map $F:L\to L$ may not be well
defined on the whole $L.$ To overcome this difficulty just set
$$D=\{(\xi,t):\xi\in\T^{3},t\in \mathcal{L}(E^u_B,E^s_B\oplus E^c_B), \|t\|\le 1\}$$
and from the above theorem we have
\begin{cor}\label{sectiondisc}
Assume that for some $r, B$ and $k$  we have
$$l_\xi(F,g_{B,k})<1,\,\,\,\,\,\,\,\,\,\,\,
l_\xi(F,g_{B,k})\left(\tau(g_{B,k})\right)^r<1.$$
 Then, there exists
$\mathcal U(g_{B,k})$ such that for any $g\in \mathcal{U}(g_{B,k})$
of class $C^\infty$ we have that the associated map $F_g:D\to D$ is
well defined, $l_\xi(F)<1$ and $l_\xi(F)\tau(g)^r<1.$ In particular
there exists a unique invariant section for $F_g$ in $D$ and it is of class
$C^r.$

\end{cor}

\begin{obs}\label{crbundle}
Observe that if $\sigma:\T^3\to L$ is an invariant section by $F$,
i.e., $F\circ \sigma=\sigma\circ g$ then it holds that
$graph(\sigma(\xi))=E^u_g(\xi).$ So, in order to find the
differentiability class we will apply the $C^r$ Section Theorem to
our $F:L\to L$ over $g.$
\end{obs}

\begin{obs}

If we use the $C^{r}$-section theorem to calculate the
differentiability of the unstable vector bundle of the Anosov system
induced by $B$, then we will have differentiability less than
$C^{3}$: Let $r=3$, then  compute
$l_{\xi}\tau_{\xi}^{r}=\frac{\lambda_c}{\lambda_u}\frac{1}{\lambda_s^{3}}=\frac{\lambda_c}{\lambda_s^{2}}>1$.
Moreover, the last estimate shows that in order to have proximity to
$C^{3}$ differentiability we must find linear Anosov systems with
$\lambda_s$ close to $\lambda_c$. This will be done in Section
\ref{fam}

\end{obs}

Through the rest of this subsection and to avoid notation we set
$g=g_{B,k}.$ We want to estimate  $l_{\xi}(F,g)$ and $\tau_{\xi}(g)$
for the graph transform $F$ associated to $g=g_{B,k}$. Recall that
the differential of $g$ in the decomposition $E^s_B\oplus
E^c_B\oplus E^u_B$ is given
 by:

\begin{eqnarray*}
&&dg_{\xi}= \left( \begin{array}{ccc}
\lambda_s & 0 & 0 \\
0 & \lambda_c & 0 \\
0 & 0 & \lambda_u \end{array} \right)\,\,\, \mbox{ for
}\xi\in\T^{3}\setminus B(p,\rho)\\&&\mbox{and }\\&& dg_{\xi}= \left(
\begin{array}{ccc}
\lambda_s + Z(z)(\beta(r)+\beta'(r)2x^2) & Z(z)\beta'(r)2xy & Z'(z)\beta(r)x \\
Z(z)\beta'(r)2xy& \lambda_c + Z(z)(\beta(r)+\beta'(r)2y^2) & Z'(z)\beta(r)y \\
0 & 0 & \lambda_u \end{array} \right)\\&& \mbox{ for }\xi \in
B(p,\rho).
\end{eqnarray*}

Set  $T_{\xi}=dg_{\xi/E^s_B\oplus E^c_B}.$
\begin{lema}\label{lip}
With the above notations we have
$$ l_\xi=l_{\xi}(F)\le\frac{\|T_{\xi} \|}{\lambda_u}.$$
Moreover the following estimations hold:
\begin{enumerate}\renewcommand{\theenumi}{\roman{enumi}}
\item\label{lfuera} For $\xi\notin B(p,\rho)$ we have $l_{\xi}\le\frac{\lambda_c}{\lambda_u}$
\item\label{ldentro} For $\xi\in B(p,\rho)$ we have $l_{\xi}<\frac{\lambda_c+Z(z)\beta(r)+k}{\lambda_u}$
\end{enumerate}
In particular $l_\xi(F)<1$ for all $\xi\in\T^3$ (if $k$ is small).
\end{lema}
\begin{proof}
If we write
\begin{eqnarray*}
dg_\xi=\left(\begin{array}{cc} T_\xi & A_\xi\\ 0 & \la_u
\end{array}\right)
\end{eqnarray*}
then it is not difficult to see that
$$F(\xi,t)(v)=\frac{1}{\la_u} \left(T_\xi(t(v))+A_\xi
v\right)$$ and therefore
$$\|F(\xi,t_1)-F(\xi,t_2)\|\le\frac{\|T_\xi\|}{\la_u}\|t_1-t_2\|$$
which implies $ l_{\xi}\le\frac{\|T_{\xi} \|}{\lambda_u}.$ Since for
$\xi\notin B(p,\rho)$ it holds that $\|T_\xi\|=\la_c$ we obtain
\eqref{lfuera}.

In order to prove \eqref{ldentro}, set $T_{\xi}=D+S_{\xi}$ where $D=
\left( \begin{array}{ccc}
\lambda_s & 0  \\
0 & \lambda_c \\
\end{array} \right)$ and
$$S_{\xi}= \left(
\begin{array}{cc}
 Z(z)(\beta(r)+\beta'(r)2x^2) & Z(z)\beta'(r)2xy \\
Z(z)\beta'(r)2xy&Z(z)(\beta(r)+\beta'(r)2y^2)  \end{array} \right)$$
Observe that $S_\xi$ is selfadjoint and has eigenvectors (when
$\xi\neq p$) $(x,y)$,$(-y,x)$ and eigenvalues
\begin{equation}\label{eigen}
\lambda_1=Z(z)(\beta(r)+2\beta'(r)r)\;\;\;\;\;\;\;\;\lambda_2=Z(z)\beta(r).
\end{equation}
When, $\xi=p$ then $S_\xi=Z(0)\beta(0)Id.$ From the definition of
$g$ (recall Lemma \ref{mod} and \eqref{g}) we have
$-k<\lambda_1<\lambda_2<\beta(0)$ and
$\lambda_2>0,\lambda_2-\lambda_1<k$. Then, $\|S_\xi\|\le
\max\{|\la_1|,|\la_2|\}\le \la_2+k=Z(z)\beta(r) +k$ and so $\|
T_{\xi}\|\le \lambda_c+\lambda_2+k.$

\end{proof}

\begin{lema}\label{tau}
Let $\la_1=\la_{1,g}:\T^3\to \R$ be the function defined by:
$\la_{1,g}(\xi)=0$ if $\xi\notin B(p,\rho)$ and
$\la_{1,g}(\xi)=Z(z)(\beta(r)+2\beta'(r)r)$ for $\xi\in B(p,\rho).$
Then, we have
$$\|(dg_\xi)^{-1}\|=\tau_{\xi}=\tau_\xi(g)\le\frac{1}{\lambda_s+\lambda_{1,g}(\xi)}$$
\end{lema}

\begin{proof}
Write \begin{eqnarray*} dg_\xi=\left(\begin{array}{cc} T_\xi &
A_\xi\\ 0 & \la_u
\end{array}\right).
\end{eqnarray*}
Then
$$(dg_\xi)^{-1}=\left(\begin{array}{cc} T_\xi^{-1} & -\la_u^{-1}T_\xi^{-1}A_\xi\\ 0 & \la_u^{-1}
\end{array}\right).$$
Since $\|A_\xi\|$ is small, $\la_u^{-1}<1$ and $\|T_\xi^{-1}\|\ge 1$
it follows
$$\tau_\xi\le \|T_\xi^{-1}\|.$$
So we want to  estimate $\| (T_{\xi})^{-1}\|$. If $\xi\notin
B(p,\rho)$ then
$$\|T_\xi^{-1}\|=\frac{1}{\la_s}=\frac{1}{\lambda_s+\lambda_1(\xi)}.$$
If $\xi=p$ then $$T_p=\left(\begin{array}{cc} \la_s+Z(0)\beta(0) &
0\\0&\la_c+Z(0)\beta(0)
\end{array}\right)$$ and so $$\|T_p^{-1}\|=\frac{1}{\la_s+Z(0)\beta(0)}=\frac{1}{\lambda_s+\lambda_1(p)}.$$

For $\xi\in B(p,\rho), \xi\neq p$ write
$$T_{\xi}=C_{\xi}+\widetilde{S}_{\xi}$$ where
$$
C_\xi=\left(
\begin{array}{cc}
  \lambda_s-\lambda_c & 0 \\
  0 & 0 \\
  \end{array}
\right)\mbox{ and } \widetilde{S}_{\xi}=\left(
\begin{array}{cc}
 Z(z)(\beta(r)+\beta'(r)2x^2)+\la_c & Z(z)\beta'(r)2xy \\
Z(z)\beta'(r)2xy&Z(z)(\beta(r)+\beta'(r)2y^2)+\la_c  \end{array}
\right).$$ The selfadjoint $\widetilde{S}_\xi$ map has eigenvectors
$(x,y),(-y,x)$ associated to eigenvalues $\lambda_1+\lambda_c$ and
$\lambda_2+\lambda_c$ where $\la_1,\la_2$ are as in \eqref{eigen}.

Let $\mathcal{E}$ the elipse with axis in the $(x,y)$ direction and
$(-y,x)$ direction, with v\'{e}rtices of norm
$\frac{1}{\lambda_2+\lambda_c}$ and $\frac{1}{\lambda_1+\lambda_c}$
respectively. We have $S_{\xi}(\mathcal{E})=S^{1}$ (the unit
circle). Thus $$T_{\xi}(\mathcal{E})\subset\left\{v:
1-\frac{\la_c-\la_s}{\la_c+\la_1}\le \|v\|\le
1+\frac{\la_c-\la_s}{\la_c+\la_1}\right\}$$. Setting
$R=1-\frac{\lambda_c-\lambda_s}{\lambda_c+\lambda_1}=\frac{\lambda_s+\lambda_1}{\lambda_c+\lambda_1}$,
we have that $$T_{\xi}^{-1}(\{v:\|v\|=R\})\subset
int(\mathcal{E})\subset
\left\{v:\|v\|\le\frac{1}{\lambda_1+\lambda_c}\right\}$$. Then
$$\|(T_{\xi})^{-1} \|\leq
\frac{1}{R}\frac{1}{\lambda_1+\lambda_c}=\frac{1}{\lambda_s+\lambda_1}$$.
\end{proof}

\subsection{A special family of linear Anosov diffeomorphism on $\T^3$.}\label{fam}

In order to construct elements  with $E^u$ bundle of class $C^{r}$
with $r$ close to $3$  we have seen that we need $B\in Sl(3,\Z)$
with eigenvalues $\lambda_s$ and $\lambda_c$ arbitrary close. For
this we will find a special family of matrices in $SL(3,\Z).$

Let us begin with the following family $\mathcal{J}=\{M_a\}_{a\in \N
\setminus \{0,1,2\}}$  of matrices in $SL(3,\Z)$ (inspired form the
one in \cite{McS}):
\begin{equation}\label{Ma}
 M_a= \left( \begin{array}{ccc}
  0 & -1 & 0 \\
  1 & a^{2}-1 & a \\
  0 & a^{3}+a & 1 \\
\end{array} \right)
\end{equation}

\begin{lema}\label{root} For every $a\in \N \setminus
\{0,1,2\}, M_a$ has eigenvalues $\alpha_a,\beta_a,\gamma_a$ such
that
$$\alpha_a <\frac{-a^{2}}{3}<-1<\beta_a<0<a^{2}<\gamma_a.$$ Furhtermore, we have
\begin{equation}\label{condroot}
-\frac{2a^2}{3}<\alpha_a<-\frac{a^{2}}{3}\,\,\,\,\,\mbox{ and
}\,\,\,\,\,\,\,\,a^2<\gamma_a<2 a^{2}.
\end{equation}

\end{lema}

\begin{proof}
The characteristic polinomial of $M_a$ is given by
$P_a(\lambda)=-\lambda^{3}+a^{2}\lambda^{2}+a^{4}\lambda+1$. The
derivative of $P_a$ is
$P_a'(\lambda)=-3\lambda^{2}+2a^{2}\lambda+a^{4}$ and has one
negative root $\la=\frac{-a^{2}}{3}$ and a positive one $\la=a^{2}$.
On the negative root of $P'_a$ the polynomial $P_a$ has relative
minimum, and on the  positive root where there is a relative maximum
of $P_a$. The value of $P_a$  on such roots are:
$$P_a\left(\frac{-a^{2}}{3}\right)=\frac{-5a^{6}}{27}+1<0\,\,\,\,\,\,\mbox{ and }\,\,\,\,\, P_a(a^{2})=a^{6}+1>0$$
Thus, $P_a(\la)$ is as in Figure \ref{pol} and the eigenvalues of
$M_a$ (i.e. the roots of $P_a(\la))$ satisfies
$$\alpha_a <\frac{-a^{2}}{3}<\beta_a<0<a^{2}<\gamma_a.$$
For the proof of the other inequalities in  \eqref{condroot} just do
some computations:
$$P_a\left(-\frac{2a^2}{3}\right)=\frac{2}{3^3}a^6+1>0\,\,\,\,\mbox{ and }\,\,\,\, P_a(2a^2)=-2a^6+1<0.$$
\end{proof}

\begin{figure}[ht]\begin{center}

\psfrag{al}{$\alpha_a$} \psfrag{be}{$\beta_a$}
\psfrag{ga}{$\gamma_a$} \psfrag{a2}{$a^2$}
\psfrag{a2/3}{$-\frac{a^2}{3}$}

\includegraphics[height=6cm]{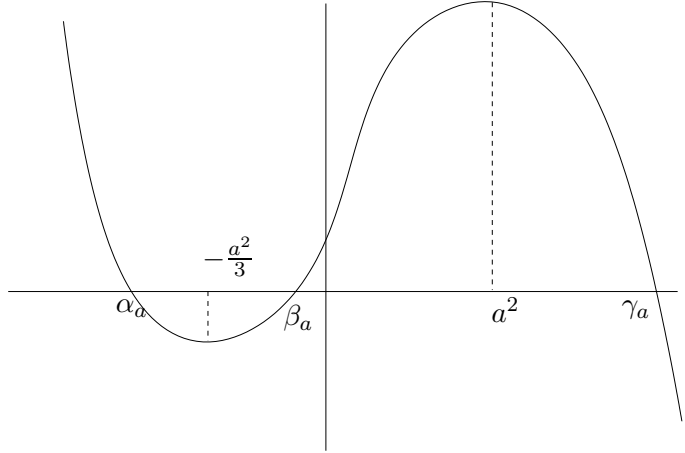}
\caption{The graph of $P_a(\la)$}\label{pol}
\end{center}\end{figure}

We are ready to define our special family of linear Anosov maps:
\begin{equation}\label{family}
\mathcal{I}=\left\{B_a=\left(M_a^2\right)^{-1}: M_a\in \mathcal{J},
a\in\N\setminus\{0,1,2\}\right\}
\end{equation}
Notice that $B_a\in SL(3,\Z)$ and the eigenvalues of $B_a$ are the
inverse of the square of the eigenvalues of $M_a$ and we have:
$$\frac{1}{4a^4}<\frac{1}{\gamma_a^2}<\frac{1}{a^4}<\frac{9}{4a^4}<\frac{1}{\alpha_a^2}
<\frac{9}{a^4}<1<\frac{1}{\beta_a^2}.$$ We summarize this in the
following
\begin{cor}
For $B_a \in\mathcal{I}$ the following holds:
\begin{enumerate}\renewcommand{\theenumi}{\roman{enumi}}
\item $B_a\in Sl(3,\Z)$ and has eigenvalues
$0<\lambda_{s}(a)<\lambda_{c}(a)<1<\lambda_{u}(a)$.

\item For every $a\in\N\setminus\{0,1,2\}$ we may write
\begin{equation}\label{ka}
\lambda_{s}(a)=K_a\frac{1}{a^{4}}\,\,\,\,\,\,\mbox{ and
}\,\,\,\,\,\,\,\,\lambda_{c}(a)=K'_a\frac{1}{a^{4}}
\end{equation}
 where
$\displaystyle{\frac{1}{10}<K_a<K'_a<10}$. In particular
$\displaystyle{\la_u(a)=\frac{a^8}{K_aK'_a}}.$

\end{enumerate}
\end{cor}

With the next result we will conclude the proof of our Main Theorem:
\begin{prop}\label{conclusion}
For each $r\in [1,3)$ there exists $B_a\in \mathcal{I}$ such that
for $g_a=g_{B_a,k}$ as defined in \eqref{id} and \eqref{g} with $k$
sufficiently small the following holds: for the map
$F=F_{B_a,g_a}:L\to L$ as defined in \eqref{fibrado} and
\eqref{grafico} and $l_\xi(F),\tau_\xi(g_a)$ as defined in Theorem
\ref{crsection} we have:
$$l_\xi(F)(\tau_\xi(g_a))^r<1\,\,\,\mbox{ for all}\,\,\,\xi\in\T^3.$$
\end{prop}

\begin{proof}
For the sake of simplicity, for $\xi\in\T^3$  set
$l_{\xi,a}=l_\xi(F_{B_a,g_a})$ and $\tau_{\xi,a}=\tau_\xi(g_a).$

Fix $r, 1\le r<3.$ It is enough to prove the proposition to show
that
$$\lim_{a\to\infty} l_{\xi,a}\tau_{\xi,a}^r=0$$
uniformly on $\xi\in\T^3.$ To do so, from Lemmas \ref{lip} and
\ref{tau}, we have for $\xi\notin B(p,\rho)$:
\begin{equation}\label{liptau}
l_{\xi,a}\tau_{\xi,a}^r=\frac{\lambda_c(a)}{\lambda_u(a)}\frac{1}{\lambda_s(a)^{r}}=
\frac{\lambda_c(a)^{2}}{\lambda_s(a)^{r-1}} = \frac{K'_a
a^{4(r-1)}}{K_aa^{8}}\le 100\frac{a^{4(r-1)}}{a^{8}}
\end{equation}
and for $\xi\in B(p,\rho):$
\begin{eqnarray*}
l_{\xi,a}\tau_{\xi,a}^r&=&\frac{\lambda_c(a)+Z(z)\beta(r)+k}{\lambda_u(a)}
\left[\frac{1}{\lambda_s(a)+\lambda_{1,g_a}(\xi)}\right]^{r} \\&&\\
&=&\frac{1}{\lambda_u(a)}\left[
\frac{\lambda_c(a)+\lambda_{1,g_a}(\xi)}{(\lambda_s(a)+\lambda_{1,g_a}(\xi))^{r}}
+\frac{k+Z(z)\beta(r)-\lambda_{1,g_a}(\xi)}{(\lambda_s(a)+\lambda_{1,g_a}(\xi))^{r}}
\right]
\end{eqnarray*}
Since $Z(z)\beta(r)-\la_{1,g_a}(\xi)\le 2k$ we have
\begin{eqnarray*}
l_{\xi,a}\tau_{\xi,a}^r&\le&\frac{1}{\lambda_u(a)}\left[
\frac{\lambda_c(a)+\lambda_{1,g_a}(\xi)}{(\lambda_s(a)+\lambda_{1,g_a}(\xi))^{r}}
+\frac{3k}{(\lambda_s(a)+\lambda_{1,g_a}(\xi))^{r}}  \right] \\&&\\
&\le& \frac{1}{\lambda_u(a)}\left[
\frac{\lambda_s(a)+\lambda_{1,g_a}(\xi)+
(\lambda_c(a)-\lambda_s(a)+3k)}{(\lambda_s(a)+\lambda_{1,g_a}(\xi))^{r}}
\right]
\end{eqnarray*}
We may assume, for fixed $a$ that $3k<\la_s(a)<10\frac{1}{a^4}.$
From the fact that $0<\la_c(a)-\la_s(a)<10\frac{1}{a^4}$ and also
that $\la_{1,g_a}(\xi)\ge -k$ we have
\begin{eqnarray*}
l_{\xi,a}\tau_{\xi,a}^r&\le&\frac{1}{\lambda_u(a)}\left[ \frac{\lambda_s(a)+\lambda_{1,g_a}(\xi)+ 20\frac{1}{a^4}}{(\lambda_s(a)+\lambda_{1,g_a}(\xi))^{r}} \right]\\
&\le& \frac{1}{\lambda_u(a)}\left[
\frac{1}{(\lambda_s(a)+\lambda_{1,g_a}(\xi))^{r-1}} +
\frac{20}{a^4(\lambda_s(a)+\lambda_{1,g_a}(\xi))^{r}} \right]\\
&\le&\frac{1}{\lambda_u(a)}\left[ \frac{1}{(\lambda_s(a)-k)^{r-1}} +
\frac{20}{a^4(\lambda_s(a)-k)^{r}} \right]\\
&\le& \frac{100}{a^8}\left[\frac{2}{(\la_s(a))^{r-1}}+\frac{40}{a^4(\la_s(a))^r}\right]\\
&\le& \frac{100}{a^8}\left[8a^{4(r-1)}+40a^{4(r-1)}\right]\\
&\le& 10^4\frac{a^{4(r-1)}}{a^8}
\end{eqnarray*}
From this and \eqref{liptau} and taking into account that $1\le r<3$
we have for $a\in\N$ large enough that
$$l_{\xi,a}\tau_{\xi,a}^r<1$$
for any $\xi\in\T^2.$ This completes the proof of the proposition.

\end{proof}

We can conclude the proof of our Main Theorem: let $r,1\le r<3$ and
choose $B_a\in\mathcal{I}$ and $g_{B_a,k}$ from the above
Proposition. From Corollary \ref{sectiondisc} we find $\mathcal
U(g_{B_a,k})$ and we choose $g\in\mathcal U(g_{B_a,k})$ of class
$C^\infty$ and having a homoclinic intersection associated to the
fixed point $p$ of unstable index two. From Theorem \ref{crsection},
Corollary \ref{sectiondisc} and remark \ref{crbundle}  the unstable
foliation $\F^u_g$ is of class $C^r$ and so, by remark \ref{dif} the
induced map $f=f_g:\T^2\to\T^2$ is of class $C^r.$ Finally, Theorem
\ref{maintopologico} implies our Main Theorem.

\end{document}